\documentclass[12pt,epsfig,amsfonts]{amsart}
\usepackage{amsmath,amsthm,amssymb,amscd,epsfig}
\usepackage{graphicx}

\newcommand{\R}{\mathbb R}
\newcommand{\N}{\mathbb N}
\newcommand{\Z}{\mathbb Z}

\begin{document}
\date{\today}



\newtheorem*{theorema}{Theorem A}
\newtheorem*{theoremb}{Theorem B}
\newtheorem*{theoremc}{Theorem C}
\newtheorem*{theoremd}{Theorem D}

\newtheorem*{theorem}{Theorem 1}
\newtheorem{definition}{Definition}[section]
\newtheorem{prop}[definition]{Proposition}
\newtheorem{lemma}[definition]{Lemma}
\newtheorem{sublemma}[definition]{Sublemma}
\newtheorem{remark}{Remark}[section]
\newtheorem{cor}[definition]{Corollary}
\newtheorem{claim}[definition]{Claim}


\author{Samuel Senti and Hiroki Takahasi}
\address{Instituto de Matematica, Universidade Federal do Rio
de Janeiro, C.P. 68 530, CEP 21945-970, R.J., BRASIL}
\email{senti@im.ufrj.br}
\address{Department of Mathematics, Keio University, 
Yokohama 223-8522, JAPAN}
 \email{hiroki@math.keio.ac.jp}
\thanks{}
\title[Equilibrium measures for the H\'enon map at the first bifurcation] {Equilibrium measures for the H\'enon map at the first bifurcation: uniqueness and geometric/statistical properties}
\begin{abstract}
For strongly dissipative H\'enon maps at the
first bifurcation parameter where the uniform hyperbolicity is destroyed by the formation of tangencies inside the limit set, we establish a thermodynamic formalism, i.e., prove the existence and uniqueness of an invariant probability measure which minimizes the free energy associated with a non continuous geometric potential $-t\log J^u$, where $t\in\mathbb R$ is in a certain large interval and $J^u$ denotes the Jacobian in the unstable direction. We obtain geometric and statistical properties of these
measures.

\end{abstract}
\maketitle

\section{Introduction}
It is a well-known fact that unfoldings of non-transverse intersections between stable and unstable manifolds unleash surprisingly rich arrays of complicated behaviors (see, e.g., \cite{PalTak93} and the references therein).
Advancing our knowledge of such complexities is essential for understanding the realm of dynamics beyond uniform hyperbolicity. 

In dimension two, an important role is played by the H\'enon family
\begin{equation}\label{henon}
f_{a}\colon(x,y)\mapsto (1-ax^2+\sqrt{b}y, \pm\sqrt{b}x),\ \
0<b\ll1.\end{equation} 
Indeed, a perturbation of this family is embedded in generic unfoldings of quadratic homoclinic tangencies associated with dissipative saddles of surface diffeomorphisms \cite{MorVia93,PalTak93}. Hence, a thorough study of the H\'enon family should provide a general account on complexities unleashed by homoclinic tangencies in dimension two.

Another important feature of the H\'enon family is that it describes a transition from Smale's horseshoe to the strange attractors of Benedicks $\&$ Carleson \cite{BenCar91}.
For sufficiently large $a$, the non-wandering set of $f_a$ is a uniformly hyperbolic horseshoe \cite{DevNit79}. 
As $a$ decreases, the stable and unstable directions get increasingly confused, until one reaches the {\it first bifurcation parameter} $a^*$ near $2$.
At $a=a^*$ the horseshoe undergoes a homoclinic (or heteroclinic) bifurcation, i.e., $\{f_a\}$ generically unfolds a quadratic tangency at $a=a^*$ between stable and unstable manifolds of the two fixed saddles  \cite{BedSmi04,BedSmi06} (see FIGURE 1).
On the other hand, close to and at the left of $a^*$ there exists a positive measure set of $a$-values corresponding to maps which admit nonuniformly hyperbolic strange attractors \cite{BenCar91}. Despite the importance of this transition, many of its aspects are poorly understood, apart from a few partial results \cite{Rio01, Tak12}. 

In this paper we study the dynamics of $f_{a^*}$ from the viewpoint of ergodic theory and thermodynamic formalism. Write $f$ for $f_{a^*}$, and let $\Omega$ denote the non-wandering set of $f$. This set is closed, bounded and hence compact.
Let $\mathcal M(f)$ denote the space of all $f$-invariant Borel probability measures endowed with the topology of weak convergence. 
For a potential function $\varphi: \Omega\to \mathbb R$ the associated (minus of the) free energy function $F_\varphi\colon\mathcal M(f)\to\mathbb R$ is given by
$$F_\varphi(\mu):= h(\mu)+\mu(\varphi),$$
where $h(\mu)$ denotes the entropy of $\mu$ and $\mu(\varphi)=\int\varphi d\mu$. An \emph{equilibrium measure} for the potential $\varphi$ is a measure $\mu_\varphi\in\mathcal M(f)$ which maximizes $F_{\varphi}$, i.e.
$$F_\varphi(\mu_\varphi)=
\sup\{F_\varphi(\mu)\colon\mu\in\mathcal M(f)\}.$$
The main example of potential functions we are concerned with is the family of potentials
$$\varphi_t:=-t\log J^u\quad t\in\mathbb R,$$ where $J^u$ denotes the Jacobian along the \emph{unstable direction} which is defined as follows. At a point $z\in \Omega$, let $E_z^u$ denote the one-dimensional subspace such that
\begin{equation}\label{eu}
\varlimsup_{n\to\infty}\frac{1}{n}\log\|D_zf^{-n}|E_z^u\|<0.\end{equation}
Since $f^{-1}$ expands area, $E^u_z$ is unique when it makes sense. We call $E^u_z$ the {\it unstable direction at $z$} and define $J^u(z):=\Vert D_zf|E^u_z\Vert$.
It was proved in  \cite[Proposition 4.1]{SenTak11} that $E^u$ makes sense for all $z\in\Omega$,
and is continuous except at the fixed saddle $Q$ near $(-1,0)$.

\begin{figure}
\begin{center}
\includegraphics[height=6.5cm,width=14cm]
{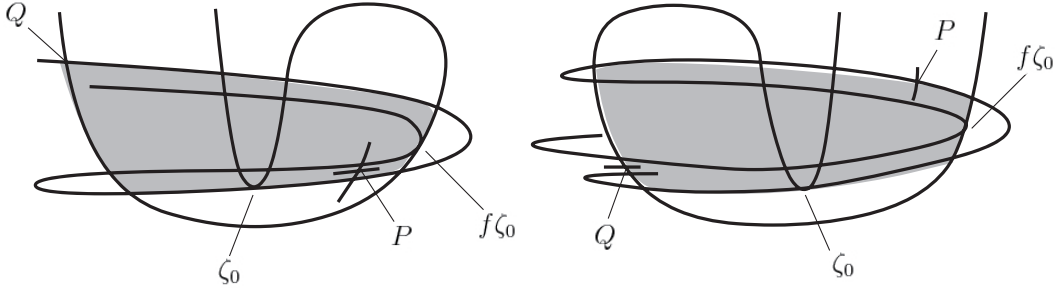}
\caption{Manifold organization for $a=a^*$. There exist two
hyperbolic fixed saddles $P$, $Q$ near $(1/2,0)$, $(-1,0)$
correspondingly. In the orientation preserving case (left),
$W^u(Q)$ meets $W^s(Q)$ tangentially. In the orientation reversing
case (right), $W^u(P)$ meets $W^s(Q)$ tangentially. The shaded
regions represent the region $R$ (see Sect.\ref{family}).}
\end{center}
\end{figure}

The (non-uniform) expansion along the unstable direction is responsible for the chaotic behavior.
Therefore, information on the dynamics of $f$ as well as the geometry of $\Omega$ is obtained
by studying equilibrium measures for $\varphi_t$, and the associated pressure function
$t\in\mathbb R\mapsto P(t)$, where
$$P(t):=\sup\{F_{\varphi_t}(\mu)\colon
\mu\in\mathcal M(f)\}.$$
Since $\varphi_t$ is merely bounded measurable, the existence of equilibrium measures for $\varphi_t$, let alone the uniqueness, is an issue.
The existence was studied in \cite{SenTak11}.
We are now concerned with the existence and uniqueness of equilibrium measures for $\varphi_t$, and their geometric and statistical properties.

\begin{theorema}
For any bounded interval $I\subset(-1,\infty)$ there exists $b_0>0$ such that if $0<b<b_0$, then for all $t\in I$ there exists a unique equilibrium measure for $\varphi_t$.
\end{theorema}

Several remarks are in order on Theorem A.
 Since entropies of invariant probability measures are written as linear combinations of the entropies of the ergodic components, and the same property holds for unstable Lyapunov exponents, the equilibrium measures in Theorem A must be ergodic. In addition, from our construction, they are supported on $\Omega$, i.e., give positive weight to any open set intersecting $\Omega$.

It was proved in \cite[Theorem]{SenTak11} that equilibrium measures for $\varphi_t$ exist for all negative $t$ and some (many) positive $t$.
We cannot rule out the possibility of the coexistence of multiple equilibrium measures for $t\leq -1$,
as is the case for the Chebyshev quadratic polynomial $x\in[-1,1]\to 1-2x^2$. This is the reason why assume $t>-1$.

There are still few results concerning the thermodynamics of the H\'enon maps. All currently known results in this direction are concerned with positive Lebesgue measure sets of parameters (close to but not containing $a^*$)  for which the corresponding maps exhibit {\it strange attractors} \cite{BenCar91,Ber09, Ber12,MorVia93,WanYou01}. For these parameters, SRB measures are constructed and shown to be unique in \cite{BenYou93} (see also \cite{Ber09}).
In our terms, these measures are
equilibrium measures for $\varphi_t$ with $t=1$.
The existence of equilibrium measures for continuous potentials is established in \cite{WanYou01}, and in particular, measures of maximal entropy exist. These are equilibrium measures for $\varphi_t$ with $t=0$. 
The uniqueness of measures of maximal entropy for a positive Lebesgue measure set of parameters is proved in \cite{Ber12}. The existence of equilibrium measures for $\varphi_t$ with $t$ other than $0,1$ is not known.

The construction used in the proof of Theorem A allows us to characterize the Hausdorff dimension of a (one dimensional) unstable slice of $\Omega$ as the first zero of the pressure (see also \cite{LepRio09,ManMac83,UrbWol08}).   
Given a $C^1$ one-dimensional submanifold $\gamma$ of $\mathbb R^2$ and $p\in(0,1]$, 
 the Hausdorff $p$-measure of a set $A\subset\gamma$ is given by
$$m_{p}(A)=
\lim_{\delta\to0}\left(\inf\sum_{U\in\mathcal
U } \ell(U)^p \right).
$$ 
Here, $\ell$ denotes the diameter with respect to the induced metric
on $\gamma$,  and the infimum is taken over all coverings $\mathcal U$ of $A$ by open sets in $\gamma$ with diameter $\leq\delta$.
 The Hausdorff dimension of $A$ on $W^u(P)$, simply denoted by $\dim^u_H(A)$, is the unique number in $[0,1]$ such that
$$
\dim^u_H(A)=
\sup\{p\colon m_p(A)=\infty\}=\inf\{p\colon m_p(A)=0\}.
$$

The pressure function $t\mapsto P(t)$ is convex, and so continuous.
One has $P(0)>0$, and Ruelle's inequality \cite{Rue78} gives $P(1)\leq0$. Since $f$ has no SRB measure \cite{Tak12}, $P(1)<0$ holds. Hence the equation $P(t)=0$ has a unique solution in $(0,1)$, which we denote by $t^u$.

\begin{theoremb}
For any open set $\gamma$ in the unstable manifold of the fixed saddles of  with $\gamma\cap \Omega\neq\emptyset$, we have ${\rm
HD}(\gamma\cap \Omega)=t^u$. In addition, $t^u\to1$ as $b\to0$.
\end{theoremb}

Our results are similar in spirit to the ones of Leplaideur and Rios \cite{LepRio05,LepRio09},
in which a thermodynamic formalism for certain horseshoes with three branches and a single orbit of tangency was established (also see \cite{Lep11}).
Certain hypotheses in \cite{LepRio05,LepRio09} on expansion/contraction rates and curvatures of invariant manifolds near the tangency
are no longer true in our setting due to its strong dissipation.
Our approach here is to take advantage of this strong dissipation,
as in the study of H\'enon-like systems \cite{BenCar91,BenYou00,MorVia93,WanYou01}.

The construction and study of many relevant invariant measures can be carried out on the symbolic level, when a coding of the orbits into symbolic sequences is available.
For uniformly hyperbolic systems, Markov partitions are used to code orbits with symbolic sequences over a finitely alphabet. The existence and uniqueness of equilibrium measures for H\"older continuous potentials were established in \cite{Bow75,Ru78,Sin72}.
However, the map $f$ lacks such a nice partition. Indeed, the natural partition of $\Omega$ into the ``left" and the ``right" of the point of tangency near the origin, constructed in \cite{SenTak11} 
only defines a semi-conjugacy between $f|\Omega$ and the full shift on two symbols. In order to avoid the discontinuity of $\varphi_t$ at $Q$, we must consider a (non-compact) subset of $\Omega$ which does not contain $Q$. We code the dynamics on this subset with a \emph{countable} alphabet to establish the uniqueness (countable partitions were also constructed in \cite{Hoe08, LepRio09} albeit for other purposes/maps).

Our strategy for proving the uniqueness of the equilibrium measures is to construct an invariant measure as a candidate, and then show that it is indeed a unique measure which 
maximizes $F_{\varphi_t}$. The main step is to build an \emph{inducing scheme} $(S,\tau)$. Here $S$ is a countable collection of pairwise disjoint 
Borel subsets of $\Omega$ called {\it basic elements}. The union of all basic elements is denoted by $X$, and $\tau$ is the first return time to $X$, which is constant on each basic element. The inducing scheme  allows us to represent the first return map to $X$ as a countable (full) Markov shift. Under certain conditions on the potential function, which are satisfied by $\varphi_t$ with $t\in(t_-,t_+)$, where $t_-<0<t_+$ depend on $t^u$ (see \eqref{t+-} for the precise definition), one can construct a Gibbs measure in the shift space following \cite{MauUrb01, Sar03}. The interval $(t_-,t_+)$ of Theorem A can be chosen arbitrarily large in $(-1, \infty)$ since $t_-\to-1$ and $t_+\to \infty$ as $t^u$ goes to $1$ (c.f. Theorem B).
This Gibbs measure is then used to obtain a unique invariant measure for the original system which minimizes the free energy among all measures which are \emph{liftable} to the inducing scheme (i.e. those measures which can be obtained from symbolic shift invariant measures).


To show that the candidate measure is a unique equilibrium measure, one must 
show that non liftable measures (e.g., the Dirac measure at $Q$) do not maximize $F_{\varphi_t}$. This can be done in two steps. We first  show that any ergodic measure with sufficiently large entropy is liftable to the inducing scheme $(S,\tau)$.  
We then show that, with some restriction on $t$, measures with small entropy do not maximize $F_{\varphi_t}$.
In the second step we essentially use the fact that holonomy maps along stable manifolds are Lipschitz continuous.
This is false in general, but true for $f$ as explained in Remark \ref{holono}.

The construction of our inducing scheme is 
inspired by the work of Benedicks and Young \cite{BenYou00} on H\'enon-like strange attractors:
points returning to a neighborhood of the tangency too fast, for which ``long stable leaves" cannot be constructed, must be 
excluded from consideration.
As a result, each basic element of the inducing scheme constructed here is Cantor-like.
In addition, 
one must analyze its Hausdorff dimension, because Lebesgue almost every initial point diverges  to infinity under positive 
iteration \cite{Tak12}.
These factors make estimates more involved than \cite{BenYou00}.

We now move on to geometric and statistical properties. In what follows, let $\mu_t$ denote the equilibrium measure for $\varphi_t$ in Theorem A. We first give a characterization of $\mu_{t^u}$ in terms of dimension. To give a precise statement let us recall general facts on nonuniformly hyperbolic systems. Let $\mathcal M^e(f)$ denote the set of ergodic elements of $\mathcal M(f)$. Since any $\mu\in\mathcal M^e(f)$ has exactly one positive Lyapunov exponent \cite{CLR08}, 
for $\mu$-a.e. $x\in\Omega$ the set
\begin{equation}\label{unsm}
W^u(x)=\left\{y\in\mathbb R^2\colon \varlimsup_{n\to\infty}\frac{1}{n}
\log |f^{-n}x-f^{-n}y|<0\right\}
\end{equation}
is a smooth injectively immersed one-dimensional submanifold of $\mathbb R^2$ \cite{Pes76,Rue79}.
We call $W^u(x)$ the {\it unstable manifold of $x$}. Let $\{\mu_x^u\}_{x\in\Gamma}$ denote the {\it canonical system of conditional measures of $\mu$ along unstable manifolds} \cite{Rok67}: $\mu_x^u$ is a probability measure supported on $W^u(x)$ such that $x\mapsto \mu_x^u(A)$ is measurable and $\mu(A)=\int\mu_x^u(A)d\mu(x)$ for any measurable set $A$. Let ${\rm dim}(\mu_x^u)$ denote the dimension of $\mu_x^u$, namely
$$
\dim(\mu_x^u)=\inf\{\dim^u_H(X)\colon X\subset W^u(x),\mu_x^u(X)=1\}.
$$ 
Then, $\dim(\mu_x^u)$ is constant $\mu$-a.e. and this number is denoted by $\dim^u(\mu)$. We say $\mu\in\mathcal M^e(f)$ is a {\it measure of maximal unstable dimension} if
$$
\dim^u(\mu)=\sup\{\dim^u(\nu)\colon\nu\in\mathcal M^e(f)\}.
$$

\begin{theoremc}
$\mu_{t^u}$ is the unique measure of maximal unstable dimension.
\end{theoremc}

Considering the tower associated to the inducing scheme allows us to apply the result of Young \cite{You98} to deduce several statistical properties of $\mu_t$.
\begin{theoremd}
The following holds for $(f,\mu_t)$;
\begin{itemize}
\item[(1)]for any $\eta\in(0,1]$ there exists $\tau\in(0,1)$ such that for any H\"older continuous $\phi\colon \Omega\to\mathbb R$ with H\"older exponent $\eta$ and $\psi\in L^\infty(\mu_t)$, there exists a constant $C(\phi,\psi)$ such that
$$
\left| \mu_t((\varphi\circ f^n)\psi)-\mu_t(\varphi)\mu_t(\psi)\right|\leq C(\varphi,\psi)\tau^{n}\quad\text{for every }n>0;
$$
\item[(2)]
 for any H\"older continuous $\phi\colon \Omega\to\mathbb R$ with $\int\phi d\mu_t=0$, there exists $\sigma\geq0$ such that
$$
\frac{1}{\sqrt{n}}\sum_{i=0}^{n-1}\phi\circ f^i \ \longrightarrow \ \mathcal N(0,\sigma)\quad\text{ in distribution},
$$
where $\mathcal N(0,\sigma)$ is the normal distribution with mean $0$ and variance $\sigma^2$. In addition, $\sigma>0$ if and only if $\phi\neq\psi\circ g-\psi$ for any $\psi\in L^2(\mu_t)$.
\end{itemize}
\end{theoremd}

The rest of this paper consists of three sections. In Sect.2 we recall the general thermodynamical formalism for maps admitting inducing schemes from~\cite{PesSenZha11}. In Sect.3 we construct an efficient inducing scheme in the above sense. In Sect.4 we define $t_-$, $t_+$ and then check all the conditions on $\varphi_t$, $t\in(t_-,t_+)$, necessary for implementing the theory in Sect.2. This yields an $f$-invariant measure $\mu_t$ which maximizes $F_{\varphi_t}$ among all liftable measures. We show that $\mu_t$ is the unique measure which maximizes $F_{\varphi_t}$ among \emph{all} measures. This completes the proof of Theorem A. Other theorems are also proved in Sect.4.

\section{Equilibrium measures for maps admitting inducing schemes}
In this section we recall the construction of equilibrium measures for $\varphi$ developed in \cite{PesSenZha11}. The main idea is to use an inducing scheme to relate the induced system to a countable Markov shift, and construct a Gibbs measure in the symbolic space for the induced potential following \cite{MauUrb01, Sar03}. Gibbs measures have integrable inducing time and are  used to construct an equilibrium measure for the original map associated to the original potential function.

\subsection{Equilibrium states for countable Markov shifts}
Denote the set of all bi-infinite sequences over a countable alphabet $S$ by $$S^\mathbb{Z}:=\{\underline{a}:=(\dots, a_{-1},a_0, a_1,\dots)\colon a_i\in S,\ i\in\mathbb{Z}\}$$ and the (left full) shift by $\sigma:S^\mathbb{Z}\circlearrowleft$ i.e. $(\sigma({\underline a}))_i=a_{i+1}$. Denote the \emph{cylinder sets} by $$[b_i, \dots, b_j]:=\{\underline{a}\in S^\Z\colon a_k=b_k \text{ for all } i\le k\le j\}.$$ Endow $S^\mathbb Z$ with the topology for which the cylinder sets form a base. The shift $\sigma$ is continuous with respect to this topology. Denote by $\mathcal M(\sigma)$ the collection of $\sigma$-invariant Borel probability measures on $S^{\mathbb Z}$. Given a function $\Phi: S^\mathbb{Z}\to\mathbb{R}$, let
\begin{eqnarray*}
\mathcal{M}_{\Phi}(\sigma) := \{ \nu\in\mathcal M(\sigma)\colon
\nu(\Phi) >-\infty\}.
\end{eqnarray*}
The \emph{n$^{th}$variation} of $\Phi$ is defined by
\[
V_n(\Phi):= \sup_{[b_{-n+1}, \ldots,
b_{n-1}]}\sup_{\underline{a},\underline{a}'\in [b_{-n+1}, \ldots,
b_{n-1}]}|\Phi(\underline{a})-\Phi(\underline{a}')|.
\]
The function $\Phi$ has \emph{strongly summable variation} if
$$
\sum_{n\ge 1}nV_n(\Phi)<\infty.
$$
The \emph{Gurevich pressure} of $\Phi$ is defined by
\[
P_G(\Phi):=\lim_{n\rightarrow
\infty}\frac{1}{n}\log\sum_{\sigma^n(\underline{a})=\underline{a}}
\exp\left(\sum_{k=0}^{n-1}\Phi(\sigma^k(\underline{a}))\right)1_{[b]}(\underline{a}),
\]
where $b\in S$.
Since it depends only on the positive side of the sequences, one can prove (as in \cite[Theorem 1]{Sar99}) that $P_G(\Phi)$ exists and is independent of $b$ whenever the variation $$V_n^+(\Phi):=\sup_{[b_0,\ldots,b_{n-1}]}\sup_{\underline{a},\underline{a}'\in [b_0,\ldots,b_{n-1}]}|\Phi(\underline{a})-\Phi(\underline{a}')|$$ over all \emph{positive cylinders} is summable: $\sum_{n\ge 1}V_n^+(\Phi)<\infty$. Also $P_G(\Phi)>-\infty$ holds in this case.
We say $\nu_{\Phi}\in\mathcal M(\sigma)$ is a \emph{Gibbs measure for $\Phi$} if there exists a constant $C>0$ such that for any cylinder set $[b_{0}, \ldots, b_{n-1}]$ and any $\underline{a}\in [b_{0}, \ldots, b_{n-1}]$ we have
$$
C^{-1} \le \frac{\nu_{\Phi}([b_{0}, \ldots, b_{n-1}])}{\exp\left(-nP_G(\Phi)+\sum_{k=0}^{n-1}\Phi(\sigma^k(\underline{a}))\right)} \le C.
$$
Note that this definition only involves \emph{positive cylinders}. We say $\nu_{\Phi}\in\mathcal M(\sigma)$ is an \emph{equilibrium measure for $\Phi$} if
$$
h_{\nu_{\Phi}}(\sigma)+\nu_{\Phi}(\Phi) =
\sup_{\nu\in{\mathcal M}_{\Phi}(\sigma)}\{ h_{\nu}(\sigma) + \nu(\Phi) \}.
$$
The thermodynamics of the full shift $\sigma$ on the space of two-sided sequences over the countable alphabet $S$ is described in the following theorem from \cite{PesSenZha11}.
\begin{prop}\label{gibbs2}\cite
{PesSenZha11}
Let $\Phi: S^\mathbb{Z}\to\mathbb{R}$ be a potential function with
$\sup\Phi<\infty$ and strongly summable variation. Then
\begin{itemize}
\item[(a)] 
$ P_G(\Phi)=\sup_{\nu\in{\mathcal M}_{\Phi}(\sigma)}\{ h_{\nu}(\sigma) + \nu(\Phi) \};$
\item[(b)] if $P_G(\Phi)<\infty$ then there exists a unique Gibbs measure $\nu_{\Phi}$ for $\Phi$;
\item[(c)] if $h_{\nu_{\Phi}}(\sigma)<\infty$
then $\nu_{\Phi}\in\mathcal{M}_\Phi(\sigma)$ and it is the unique equilibrium measure for $\Phi$.
\end{itemize}
\end{prop}
The main idea is to reduce the problem to the (left full) shift on the set of one-sided infinite sequences $S^{\N}$ by constructing a potential function cohomologous to the given potential $\Phi$ but which depends only on the positive coordinates of any point $\underline{a}\in S^{\N}$.
The variational principle and the existence of a unique Gibbs and equilibrium measure for the one-sided shift and potential follows from \cite[Theorem 3]{Sar99},\cite[Theorem 1]{Sar03}, \cite[Theorem 1.1]{BuzSar03} (see also \cite{MauUrb01}). The statements of  Proposition~\ref{gibbs2} follow by considering the natural extension of this one-sided Gibbs and equilibrium measure.

\subsection{Gibbs and equilibrium measures for the induced map}
From now on assume that $f$ is a continuous self map of finite topological entropy of a compact metric space $M$.

\begin{definition}\label{scheme}
We say $f$ admits an \emph{inducing scheme $(S,\tau)$ of hyperbolic type}. 
if there exist a countable collection $S$ of disjoint Borel subsets of $M$ called {\it basic elements}, and an \emph{inducing time function} $\tau: S\to\N$ such that the following holds for the \emph{inducing domain} $X:=\bigcup_{J\in S}J$ and the \emph{induced map} $F:X\circlearrowleft$ defined by $F|J=f^{\tau(J)}|J$ for each $J\in S$. 

\begin{itemize}
\item[(A1)] $F(J)=f^{\tau(J)}J\subset X$ for each $J\in S$, and $f^{\tau(J)}|J$ extends to a homemorphism on $\overline J$;

\item[(A2)]  for any $\underline{a}=\{J_n\}_{n\in\mathbb Z}\in S^{\mathbb Z}$,  
the coding map
$h\colon S^{\mathbb Z}\to X^*:=\bigcup_{J\in S}\overline{J}$ given by
\begin{equation}\label{codingmap}
h(\underline{a}):=
\overline{J_{0}}\cap\left(\bigcap_{n\geq1} f^{-\tau(J_{0})}\circ \cdots\circ f^{-\tau(J_{n-1})}(\overline{J_n})\right)
\cap \left(\bigcap_{n\geq1} f^{\tau(J_{-1})}\circ \cdots\circ f^{\tau(J_{-n})}(\overline{J_{-n}})\right)
\end{equation}
is well-defined. The restriction of $h$ to $S^{\mathbb Z}\setminus h^{-1}(X^*\setminus X)$
is a measurable bijection onto $X$ for which
$F\circ h=h\circ\sigma$.

\item[(A3)] If $\nu$ is a $\sigma$-invariant Borel probability measure, then $\nu(h^{-1}(X^*\setminus X))=0$.
\end{itemize}
\end{definition}

In section~\ref{symbolcode} we construct an inducing scheme which satisfies conditions (A1)--(A3). Note that these conditions are stronger and thus imply the conditions of \cite{PesSenZha11}. This is due to the fact that the inducing scheme is constructed over a first return time and that the boundary of the elements consists of stable manifolds of the fixed point $P$.

If $f$ admits an inducing scheme $(S,\tau)$ of hyperbolic type, the \emph{induced potential} $\overline\varphi\colon X\to\R$ associated to a given potential $\varphi\colon M\to\R$ is defined by
$$\overline{\varphi}:= \sum_{i=0}^{\tau-1}\varphi\circ f^i.$$

We say the induced potential $\overline{\varphi}$ has:
\begin{itemize}
\item (strongly) summable variations if $\Phi:=\overline{\varphi}\circ h$ has (strongly) summable variations;
\item finite Gurevich pressure if $P_{G}(\Phi)<\infty$.
\end{itemize}
Let $\mathcal M(F)$ denote the set of $F$-invariant Borel probability measures on $X$ and  $\mathcal{M}_{\overline{\varphi}}(F)=\{\nu\in\mathcal{M}(F)\colon\nu(\overline{\varphi})>-\infty\}.$ 
An $F$-invariant probability measure $\nu_{\overline{\varphi}}$ is a \emph{Gibbs measure for $\overline{\varphi}$} if there exists an $\sigma$-invariant Gibbs measure $\nu_{\Phi}$ for $\Phi$ such that $\nu_{\overline\varphi}=h_*\nu_{\Phi}$. We call $\nu_{\overline{\varphi}}$ an \emph{equilibrium measure} for $\overline{\varphi}$ if $\nu_{\overline{\varphi}}
\in\mathcal{M}_{\overline{\varphi}}(F)$ and
$$h_{\overline{\varphi}}(F)+\nu_{\overline{\varphi}}(\overline{\varphi})=\sup\left\{\nu\in\mathcal M_{\overline{\varphi}}(F)\colon h_{\nu}(F)+\nu(\overline{\varphi})\right\}.$$

By (A2), $h_*$ preserves entropy, the Gibbs property and integrals of potentials for measures supported on 
$S^{\mathbb Z}\setminus h^{-1}(X^*\setminus X)$. 
Additionally, $h^{-1}(X^*\setminus X)$ does not support any 
measures by Condition (A3).
So the next statement is a direct consequence of Proposition \ref{gibbs2}.

\begin{cor}\label{inducedequilibrium}
Assume $f$ admits an inducing scheme $(S,\tau)$ of hyperbolic type and let $\varphi:M\to \R$ be a potential with
$\sup\overline{\varphi}<\infty$, strongly summable variations and finite Gurevich pressure. Then there exists a unique
$F$-invariant Gibbs measure $\nu_{\overline{\varphi}}$ for
$\overline{\varphi}$. If $h_{\nu_{\overline{\varphi}}}(F)<\infty$ then $\nu_{\overline{\varphi}}\in\mathcal{M}_{\overline{\varphi}}(F)$ and it is the unique equilibrium measure for $\overline{\varphi}$.
\end{cor}

\subsection{Candidate equilibrium measures for the original map}\label{candy}
We now use the Gibbs measure for the induced map $F$ to construct an equilibrium measure for the original map $f$. For $\nu\in\mathcal{M}(F)$ with $\nu(\tau)<\infty$, the measure given by
$$\mathcal{L}(\nu):=\frac{1}{\nu(\tau)}\sum_{k=1}^{\infty}(f^k)_*\nu|_{\{\tau< k\}}$$
is an $f$-invariant Borel probability measure. 
Let 
$$\mathcal{M}_L(f):=\{\mu\in\mathcal{M}(f)\colon \exists\ \nu\in\mathcal{M}(F) \mbox{ such that } \mathcal{L}(\nu)=\mu\}.$$
Measures in $\mathcal{M}_L(f)$ are called \emph{liftable}, and
for $\mu\in\mathcal M_L(f)$ a measure $\nu$ with $\mathcal L(\nu)=\mu$
is called a {\it lift} of $\mu$.

 Consider a potential $\varphi\colon M\to\mathbb R$, and let
\begin{equation}\label{PL}
P_L(\varphi):=\sup\{h_\mu(f)+\mu(\varphi)\colon
\mu\in\mathcal{M}_L(f)\}.
\end{equation}
We say $\mu\in\mathcal M_L(f)$ is a \emph{candidate equilibrium measure for $\varphi$} if $F_\varphi(\mu)=P_L(\varphi)$. Candidate equilibrium measures are equilibrium measures in the classical sense if
$\displaystyle{
P_L(\varphi)
=\sup_{\mu\in\mathcal M(f)}
\{h_\mu(f)+\mu(\varphi)\}}$.
Abramov's and Kac's formul{\ae} \cite[Theorem 2.3]{PesSen08} relate the entropy of $\mu$ and the integral of a potential $\varphi$ against $\mu$ to the entropy and the integral of the induced potential $\overline{\varphi}$ against a lift of $\mu$. 
Note that $F_\varphi(\mathcal{L}(\nu))=\frac{1}{\nu(\tau)}F_{\overline{\varphi}}(\nu)$ and so it is not straightforward that an equilibrium measure for $\overline{\varphi}$ projects to a candidate equilibrium measure for $\varphi$. However, this is the case for the equilibrium measure associated to the potential induced by $\varphi-P_L(\varphi)$ and the latter is cohomologous to $\varphi$. 
Observe that by \cite[Theorem 4.2]{PesSen08}, the existence of a periodic point of $F$ implies that $|P_L(\varphi)|<\infty$ whenever $\varphi$ has summable variations and finite Gurevich pressure.

We say $\overline{\varphi}$ is \emph{positive recurrent} if there exists $\eta_0>0$ such that
\begin{equation}\label{PR}
P_G(\overline{\varphi-(P_L(\varphi)-\eta)})<\infty\ \text{for all}\ 0\le\eta\le\eta_0.
\end{equation}
This condition implies positive recurrence condition in the sense of Sarig (c.f.~\cite{Sar03}). Indeed, \cite[Theorem 4.4]{PesSen08} and the continuity of $P_G(\overline{
\varphi-(P_L(\varphi)-\eta)})$ with respect to $\eta$ for a positive recurrent potential $\overline\varphi$ imply $P_G(\overline{\varphi-P_L(\varphi)})=0$.
%
%
%
This implies the existence of some $N\in\N$ such that 
$$\inf_{n\ge N}\left\lbrace
\sum_{F^nx=x}\exp\left(\sum_{i=0}^{n-1}\overline{\varphi-P_L(\varphi)}(F^ix)\right)\right\rbrace>0,
$$
which is equivalent to the positive recurrence condition of Sarig (c.f. \cite[Theorem 1]{Sar03}).

With condition (A3) we obtain the following: 

\begin{prop}\label{mainthemorem}\cite
{PesSenZha11}{\rm (Existence and uniqueness of candidate equilibrium measures)}\
Assume $f$ admits an inducing scheme $(S, \tau)$ of hyperbolic type. Let
$\varphi:M\to\R$ be such that $\sup \overline{\varphi-P_L(\varphi)}<\infty$, and that $\varphi$ has strongly summable variations, finite Gurevich pressure and is positive recurrent. Then there exists a Gibbs measure $\nu$ for $\overline{\varphi-P_L(\varphi)}$. If $h_\nu(F)<\infty$, then $\nu\in\mathcal M_{\overline{\varphi-P_L(\varphi)}}(F)$,
and $\nu$ is the unique equilibrium measure for $\overline{\varphi-P_L(\varphi)}$.
If $\nu(\tau)<\infty$,
then $\mathcal{L}(\nu)$ is the
unique candidate equilibrium measure for $\varphi$.
\end{prop}

\section{Construction of inducing scheme}


In this section we construct an inducing scheme which will be used for the proof of the theorems. 
In Sect.\ref{cor} we first state the existence of an inducing scheme with special properties (See Proposition \ref{lat}). 
After preliminary geometric considerations in Sect.\ref{family}, we construct 
in Sect. \ref{latt} and \ref{induced} a uniformly hyperbolic induced map with countably many branches. 
In Sect.\ref{symbolcode} we show how to obtain the inducing scheme from this induced map.

\subsection{Inducing scheme}\label{cor}
We start with preliminary definitions.
\begin{definition}\label{deflattice}
{\rm 
Let  $\Gamma^u$ and $\Gamma^s$ be two families of compact $C^1$ curves such that:
\begin{itemize}
\item  curves in $\Gamma^s$ are pairwise disjoint. Curves in $\Gamma^u$ are not necessarily pairwise disjoint;
\item every $\gamma^u\in\Gamma^u$ intersects every $\gamma^s\in\Gamma^s$ at exactly one point;
\item there is a minimum angle between $\gamma^u$ and $\gamma^s$ at the point of intersection;
\end{itemize}
Call the set
$$\Lambda:=\{\gamma^u\cap\gamma^s\colon
\gamma^u\in\Gamma^u,\gamma^s\in\Gamma^s\}$$
a {\it lattice} defined by the families $\Gamma^u$ and $\Gamma^s$.
}

{\rm
\begin{itemize}
\item $\Lambda'\subset\Lambda$ is a {\it u-sublattice} of $\Lambda$ if there exists ${\Gamma^u}'\subset\Gamma^u$ such that $\Lambda'=\{\gamma^u\cap\gamma^s\colon \gamma^u\in{\Gamma^u}',\gamma^s\in\Gamma^s\}$. An {\it s-sublattice of $\Lambda$} is defined similarly;
\item $Q_{\Lambda'}\subset\mathbb R^2$ is the {\it rectangle spanned by $\Lambda'$} if $\Lambda'\subset Q_{\Lambda'}$ and the boundary 
$\partial Q_{\Lambda'}$ is made up of two non-intersecting curves in ${\Gamma^u}'$ and two in $\Gamma^s$. 
\end{itemize}
}
\end{definition}
We now introduce a small constant $\varepsilon>0$ in order to quantify the proximity of $f$ to the Chebyshev quadratic polynomial $x\in[-1,1]\mapsto 1-2x^2$. Set
\begin{equation}\label{sigma12}
\sigma_1=2-\varepsilon\ \ \text{and}\ \ \sigma_2=4+\varepsilon.
\end{equation}
The next proposition states the existence of an inducing scheme with special properties.

\begin{prop}\label{lat}
For any small $\varepsilon>0$ there exists $b_0>0$ such that if $0<b<b_0$, 
there exist a closed lattice $\Lambda$ defined by families $\Gamma^u$ and $\Gamma^s$, a collection $S$ of pairwise disjoint Borel
subsets of $\Lambda$ and a function $\tau\colon S\to \mathbb N$ such that $(S,\tau)$ is an inducing scheme of $f=f_{a^*(b)}$ with the following properties:
\begin{itemize}
\item[(P1)] {\rm (Topological structure)} 
for each $J\in S$, $f^{\tau(J)}J\subset\bigcup_{J\in S} J$ and $f^{\tau(J)}\overline{J}$ is a $u$-sublattice of $\bigcup_{J\in S} \overline{J}$;
\item[(P2)]{\rm (Backward contraction)}
there exist $C>0$ and $\lambda>1$ such that  for each $\gamma^u\in\Gamma^u$, $z\in\gamma^u$ and $n>0$, $\|D_{f^{-n}z}f^n|T_{f^{-n}z}\gamma^u\|\geq C\lambda^{n};$
\item[(P3)]{\rm  (Hyperbolicity)}
\begin{itemize}
\item[(a)] for each $\gamma^u\in\Gamma^u$,
$J\in S$  and all $z\in\gamma^u\cap Q_{J}$,
$$
\sigma_1^{\tau(J)}\leq\Vert D_zf^{\tau(J) }|T_z\gamma^u\Vert\leq\sigma_2^{\tau(J)},
$$
where $\sigma_1,\sigma_2$ are the constants in \eqref{sigma12};
\item[(b)] for each $\gamma^s\in\Gamma^s$ and all $z\in\gamma^s$, 
$$
\Vert D_zf^{n}|T_z\gamma^s\Vert\leq(Cb)^{\frac{n}{2}}\quad \forall n\geq1,
$$
where $C>0$ is a constant independent of $\varepsilon$ and $b$;
\end{itemize}
\item[(P4)]{\rm  (Distortion control)} 
\begin{itemize}
\item[(a)]
for each $\gamma^u\in\Gamma^u$ and all $x,y\in\gamma\cap Q_{J}$,
$$
\log\frac{\|D_x f^{\tau(J)}|T_x\gamma^u\|}{ \|D_{y}f^{\tau(J)}|T_y\gamma^u\|}\leq C|f^{\tau(J)}x-f^{\tau(J)}y|,
$$
where $C>0$ is a constant independent of $\varepsilon$ and $b$;
\item[(b)] for each $\gamma^s\in\Gamma^s$ and all $x,y\in \Lambda\cap \gamma^s$,
$$
\|D_x f^n|T_x\gamma^u\|\leq 2\|D_y f^n|T_y\gamma^u\|\quad \forall n\geq1;
$$
\end{itemize}
\item[(P5)] 
Set $S(n):=\#\{J\in  S\colon \tau(J)=n\}$. Then
$$
\varlimsup_{n\to\infty}\frac{1}{n}\log S(n)\leq\varepsilon;
$$
\item[(P6)] any $\mu\in\mathcal M^e(f)$ with $h(\mu)\geq 2\varepsilon$ is liftable (in the sense of Sect.\ref{candy}) to the inducing scheme $(S,\tau)$.
\end{itemize}
\end{prop}

The rest of this section is entirely devoted to a proof of Proposition \ref{lat}. 
Along the way we introduce large integers $\xi$, $N$ the purpose of which is as follows:
\begin{itemize}
\item $\xi$ determines the rate of approach of points in the lattice $\Lambda$ to critical zones around
$\zeta_0$ (see \eqref{sr}). We set
\begin{equation}\label{xii}
\xi=\left[\frac{10}{\varepsilon}\right];
\end{equation}

\item $N$ determines the size of a critical region $\Theta_0$ (See Sect.\ref{latt}.)
\end{itemize}
For any given $\varepsilon$ as in the statement of Proposition \ref{lat}, we may choose sufficiently large $N$ at the expense of reducing $b$. 
Any generic positive constant which is independent of $\varepsilon$, $N$, $b$ is denoted by $C$.

\subsection{The return map}\label{family}
For the Chebyshev quadratic polynomial $x\in[-1,1]\mapsto 1-2x^2$, the first return
map to $[-1/2,1/2]$ is uniformly expanding with controlled distortions \cite{Jak78}.
We prove an analogous statement for $f$.

Recall that $P$, $Q$ denote the fixed saddles near $(1/2,0)$ and $(-1,0)$ respectively.
If $f$ preserves orientation, let $W^u=W^u(Q)$. 
If $f$ reverses orientation, let $W^u=W^u(P)$.
By a {\it rectangle} we mean any
closed region bordered by two compact curves in $W^u$ and two in the
stable manifolds of $P$, $Q$. By an {\it unstable side} of a
rectangle we mean any of the two boundary curves in $W^u$. A {\it
stable side} is defined similarly.

Denote by $\hat\alpha_0^-$ the connected component of $W^s(Q)\cap \{(x,y)\in \mathbb R^2\colon |y|\leq b^{\frac{1}{4}}\}$ containing $Q$ and by $\hat\alpha_0^+$ the connected component of $f^{-1}\hat\alpha_0^-\cap \{(x,y)\in \mathbb R^2\colon |y|\leq b^{\frac{1}{4}}\}$ which does not contain $Q$. Let $\gamma_0$ denote the compact curve in $W^u$ with endpoints in $\hat\alpha_0^-$ and $\hat\alpha_0^+$ which contains the saddle in $W^u$. Let $R$ denote the 
rectangle bounded by $f\gamma_0$ and $\hat\alpha_0^\pm$ (see FIGURE 1).
One of the unstable sides of $R$ contains the point of tangency near $(0,0)$ denoted by $\zeta_0$ and $\alpha_0^+$ denote the stable side of $R$ which contains $f\zeta_0$. Let $\alpha_0^-$ denote the other stable side of $R$.

Define a sequence $\{\tilde\alpha_n\}_{n\geq0}$ of compact curves in $W^s(P)\cap R$ inductively as follows. First, let $\tilde\alpha_0$ be the component of $W^s(P)\cap R$ containing $P$. Given $\tilde\alpha_{n-1}$, define $\tilde\alpha_n$ to be one of the two components of $f^{-1}\tilde\alpha_{n-1}\cap R$ which is at the left of $\zeta_0$. Observe that  $\{\tilde\alpha_n\}$ accumulates on $\alpha_0^-$ from the right.

For each $n\geq0$, $f^{-2}\tilde\alpha_n\cap R$ consists of four curves, two of them at the left of $\zeta_0$ and two at the right. Let $\alpha_{n+1}^-$ denote the one which is not $\tilde\alpha_{n+2}$ and is at the left of $\zeta_0$. Among the two at the right of $\zeta_0$, let $\alpha_{n+1}^+$ denote the one which is at the left of the other.  Then $\{\alpha_n^-\}$ (resp. $\{\alpha_n^+\}$) accumulates the component of $W^s(Q)\cap R$ containing $\zeta_0$ from the 
left (resp. right). Observe that $\tilde\alpha_1=\alpha_1^-$ and $\tilde\alpha_0=\alpha_1^+$.
By definition, the curves obey the following diagram
$$\{\alpha_{n+1}^-,\alpha_{n+1}^+\}\stackrel{f^2}{\to}\tilde\alpha_n
\stackrel{f}{\to}\tilde\alpha_{n-1}
\stackrel{f}{\to}\tilde\alpha_{n-2}\stackrel{f}{\to}\cdots
\stackrel{f}{\to}\tilde\alpha_1=\alpha_1^-\stackrel{f}{\to}
\tilde\alpha_0=\alpha_1^+.
$$

By a \emph{$C^2(b)$-curve} we mean a compact, nearly horizontal $C^2$ curve such that the slopes of its tangent directions are $\leq\sqrt{b}$ and the
curvature is everywhere $\leq\sqrt{b}$.

Given a $C^2(b)$-curve $\gamma$ with endpoints in $\bigcup_{n\geq1}\alpha_n^+\cup\alpha_n^-$, we define a partition $\mathcal P(\gamma)$ of $\gamma$ into $C^2(b)$-curves, by intersecting it with the countable families $\{\alpha_n^+\}$, $\{\alpha_n^-\}$ of pieces of stable manifolds. This is feasible since each of these pieces intersects $\gamma$ exactly one point (See \cite[Remark 2.4]{Tak12}).

\begin{figure}
\begin{center}
\includegraphics[height=4cm,width=10cm]{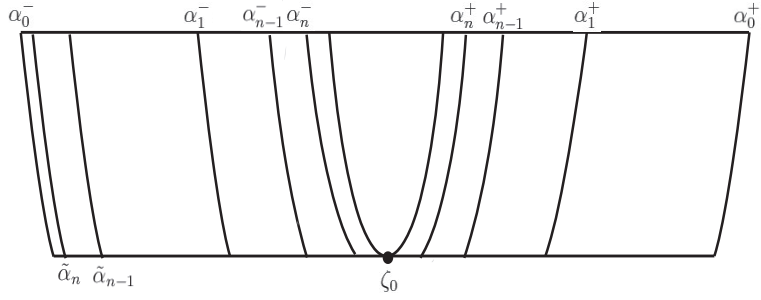}
\caption{The rectangle $R$ and the curves $\{\tilde\alpha_n\}$, $\{\alpha_n^{+}\}$, $\{\alpha_n^{-}\}$.
The $\{\tilde\alpha_n\}$ accumulate on the left stable side of $R$. Both $\{\alpha_n^{+}\}$ and $\{\alpha_n^{-}\}$ accumulate on the parabola in the stable manifold
containing the point of tangency $\zeta_0$ near $(0,0)$.}
\end{center}
\end{figure}

Let $\Theta$ denote the rectangle bordered by $\alpha_1^-$, $\alpha_1^+$ and the unstable sides of $R$. 

\begin{lemma}{\rm (\cite[Lemma 2.1]{SenTak11})}\label{component}
Each connected component of $\Theta\cap W^u$ is a $C^2(b)$-curve.
\end{lemma}

Let $\gamma$ be a connected component of $\Theta\cap W^u$.  
For each $n>1$ there is a unique element of $\mathcal P(\gamma)$ with endpoints in $\alpha_{n}^+, \alpha_{n-1}^+$ and a unique element with endpoints in $\alpha_{n-1}^-, \alpha_{n}^-$. No confusion will arise if we simplify notation by denoting both curves by $\gamma_n$.

\begin{lemma}\label{recommend}
For any component $\gamma$ of $\Theta\cap W^u$ and
each $\gamma_n\in\mathcal P(\gamma)$ the following holds:

\begin{itemize}
\item[(a)] $f^{i}\gamma_n\subset \overline{R\setminus\Theta}$ for every $1\leq i\leq n-1$;

\item[(b)] $f^n\gamma_n$ is a $C^2(b)$-curve in $\Theta$ with endpoints in the stable sides of $\Theta$
(See FIGURE 3).
\end{itemize}
\end{lemma}

\begin{proof}
We clearly have $f\gamma_n\subset \overline{R\setminus\Theta}$.
Let $1<i\leq n$.
The endpoints of $f^i\gamma_n$ are in $\tilde\alpha_{n-i+1}$, $\tilde\alpha_{n-i}$.
\cite[Lemma 2.1]{SenTak11} implies that the sets $f^i\gamma_n\cap \tilde\alpha_{n-i+1}$, 
$f^i\gamma_n\cap \tilde\alpha_{n-i}$ are singleton.
Since $f^i\gamma_n\subset R$, it follows that
$f^i\gamma_n$ is contained in the rectangle bordered by 
$\tilde\alpha_{n-i+1}$, $\tilde\alpha_{n-i}$ and the unstable sides of $R$.
Hence $f^{i}\gamma_n\subset \overline{R\setminus\Theta}$ and (a) holds.
(b) follows from \cite[Lemma 2.1]{SenTak11} which states that any  component of $\Theta\cap W^u$ is a $C^2(b)$-curve .
\end{proof}

The next lemma, the proof of which is given in Appendix A1, states that $f^n$ expands tangent vectors of $\gamma_n$ uniformly,
with controlled distortions.

\begin{lemma}\label{markov}
There exist $C>0$ and $N_0>0$ such that for any component $\gamma$
of $\Theta\cap W^u$ and each $\gamma_n\in\mathcal P(\gamma)$, $n>N_0$ we have:

\begin{itemize}
\item[(a)] for all $x\in\gamma_n$, $\sigma_1^n\leq\|D_xf^{n}|E^u_x\|\leq
\sigma_2^n$;

\item[(b)]

for all $x,y\in\gamma_n$,
$\displaystyle{\log\frac{\|D_x f^n|E^u_x\|}{\|D_{y}f^n|E^u_y\|}\leq
C|f^nx-f^ny|}.$
\end{itemize}
\end{lemma}

\begin{figure}
\begin{center}
\includegraphics[height=5cm,width=10cm]{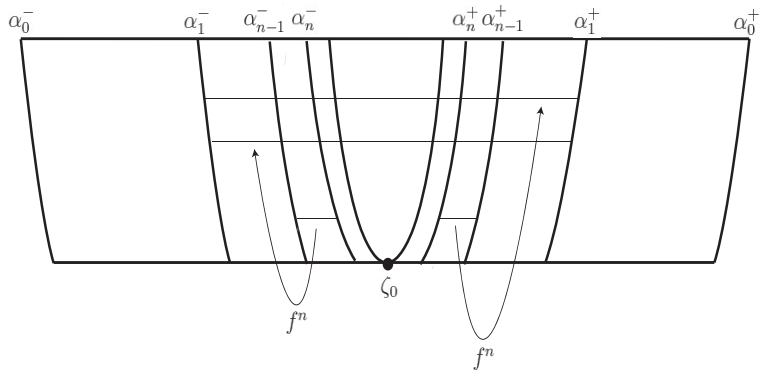}
\caption{The partition element(s) $\gamma_n$ is mapped by $f^n$ to a $C^2(b)$-curve connecting
the two stable sides of $\Theta$
(Lemma \ref{recommend}(b)).}
\end{center}
\end{figure}

Let $N\geq N_0$ (See Lemma \ref{markov}), and
let
$\Theta_0=\Theta_0(N)$ denote the rectangle bordered by $\alpha_{N}^-$,
$\alpha_{N}^+$ and the unstable sides of $\Theta$. 
Returns to the inside of $\Theta_0$ are treated by Lemma \ref{markov}.

\subsection{Construction of the lattice $\Lambda$}\label{latt}
We now construct a lattice $\Lambda$ defined by families
$\Gamma^u$, $\Gamma^s$
 of $C^1$ curves in $\Theta$.
\medskip

\noindent{\it Construction of $\Gamma^u.$}
Denote by 
$\tilde\Gamma^u$ the collection of connected components of $\Theta\cap W^u$. Define
$$\Gamma^u=\{\gamma^u\colon\text{$\gamma^u$ is the pointwise limit of
a sequence in $\tilde\Gamma^u$}\}.$$ Since elements of $\tilde\Gamma^u$ are $C^2(b)$-curves by Lemma \ref{component}, the pointwise
convergence is equivalent to the uniform convergence. Since curves in $\tilde\Gamma^u$ are pairwise disjoint,
the uniform convergence is equivalent to the $C^1$ convergence.
Hence, curves in $\Gamma^u$ are $C^1$ and the slopes of their
tangent directions are $\leq\sqrt{b}$.

\begin{remark}\label{intersec}
{\rm Several remarks are in order on what intersections are possible for curves in $\Gamma^u$
 (See FIGURE 4, and  also \cite[Lemma 4.3, Fig. 5]{BenVia06} for comparison):

\begin{itemize}
\item each $\gamma^u$-curve is the
(strictly) monotone  limit of curves in $\tilde\Gamma^u$. 
 For curves in  $\Gamma^u\setminus\tilde\Gamma^u$, this follows from the definition. 
For those in $\tilde\Gamma^u$, this follows from the Inclination Lemma (\cite[Proposition 6.2.23]{KatHas95}).
Hence, any connected component of the union of $\gamma^u$-curves contains at most two curves;

\item two intersecting $\gamma^u$-curves are tangent at every point of the intersection;

\item the backward contraction in Proposition \ref{lat}(P2) and that $f$ contracts area imply the following: if  $\gamma_1$, $\gamma_2\in\Gamma^u$ intersect each other,
then $\gamma_1\cap\gamma_2$ is connected;

\item this implies that there are at most countably many pairs that intersect each other.
\end{itemize}}

\end{remark}


\noindent{\it Construction of $\Gamma^s.$}
For each $k\geq 0$ let
$\Theta_k$ denote the rectangle bordered by $\alpha_{\xi k+N}^-$,
$\alpha_{\xi k+N}^+$ and the unstable sides of $\Theta$. Observe that
$$\Theta\supset\Theta_0\supset\Theta_1\supset\Theta_2\supset\cdots.$$

Let
$\gamma^u(\zeta_0)$ denote the lower unstable side of $\Theta$, which contains the point
$\zeta_0$ of tangency. Set
 $\Omega_0:=\overline{\gamma^u(\zeta_0)\setminus\Theta_0}$.
 For $n>0$ define the set of points whose recurrence rate to the region of tangency is slow
\begin{equation}\label{sr}
\Omega_n:=\overline{\{z\in\gamma^u(\zeta_0)\cap\Omega_{n-1}\colon f^nz\in R\setminus\Theta_{ n}\}},\end{equation} 
where $\xi$ is defined in \eqref{xii}.
Set $\Omega_\infty:=\bigcap_{n\geq0}\Omega_{n}$, which is the nonempty compact set (See Remark \ref{gap remark} below).
The fact that the rate at which the orbit of points in $\Omega_\infty$ returns close to the homoclinic tangency is slow will be fundamental in proving that the first return map to the set $\Lambda$ exhibits the properties mentioned in Proposition~\ref{lat}).

\begin{remark}\label{Komega}
{\rm Let $K=\{z\in\mathbb R^2\colon \text{$\{f^nz\}_{n\in\mathbb Z}$ is bounded}\}$.
It is immediate to check that $\Omega_\infty\subset K$.
Since $K=\Omega$ \cite[Section 3]{SenTak11}, $\Omega_\infty\subset\Omega$ holds.}
\end{remark}

By a {\it vertical $C^2(b)$-curve} we mean a compact, nearly vertical $C^2$ curve 
with endpoints
in the unstable sides of $\Theta$, and of the form
$$\{(x(y),y)\colon |x'(y)|\leq C\sqrt{b}, |x''(y)|\leq
C\sqrt{b}\}.$$
A vertical $C^2(b)$-curve $\gamma^s$ is called a {\it long stable leaf} if 
for any $x,y\in\gamma^s$, $|f^nx-f^ny|\leq \left(Cb\right)^{\frac{n}{2}}$
holds for every $n\geq0$.
The next
lemma is proved in Appendix A2.
The angle $\angle(\cdot,\cdot)$ between two one-dimensional tangent spaces is given 
by the (smaller) angle between their basis vectors.

\begin{lemma}\label{leaf}
For any $z\in\Omega_\infty$ 
there exists a unique long stable leaf $\gamma^s(z)$ through $z$. In addition, $\gamma^s(z)\subset\Theta$ and the following holds:
\begin{itemize}

\item[(a)] if $f^n(\gamma^s(z_1))\cap \gamma^s(z_2)\neq\emptyset$ for $n\geq 0$,
then $f^n\gamma^s(z_1)\subset\gamma^s(z_2)$.

\item[(b)]  $\|Df_x^n\left(\begin{smallmatrix}1\\0\end{smallmatrix}\right)\|\leq2\cdot 
\|Df_y^n\left(\begin{smallmatrix}1\\0\end{smallmatrix}\right)\|$
for all $x$, $y\in\gamma^s(z)$ and $n\geq0$; 

\item[(c)] if 
$x_1\in\gamma^s(z_1)$,
$x_2\in\gamma^s(z_2)$, then $\angle(T_{x_1}\gamma^s(z_1),T_{x_2}\gamma^s(z_2))\leq C\sqrt{b}|x_1-x_2|$. 

\end{itemize}
\end{lemma}

Define
$$\Gamma^s=\{\gamma^s(z)\colon z\in\Omega_\infty\},$$
where $\gamma^s(z)$ is the long stable leaf through $z$ in Lemma \ref{leaf}.

\subsection{Construction of an induced map on $\Lambda$}
\label{induced}
Consider the lattice $\Lambda$ defined by $\Gamma^u$ and $\Gamma^s$:
$\Lambda=\{\gamma^u\cap\gamma^s\colon
\gamma^u\in\Gamma^u,\gamma^s\in\Gamma^s\}$.
Let
$$ \mathcal W^s=\bigcup_{\gamma^s\in\Gamma^s}\gamma^s.$$

Define a first entry time $\tau\colon \Omega\to\mathbb N\cup\{\infty\}$ to $\Lambda$ by
$$\tau(z)=\inf\left(\{n>0\colon
f^nz\in\Lambda\}\cup\{\infty\}\right).$$ 
\cite[Lemma 2.2]{SenTak11} implies $\Omega\cap\mathcal W^s\subset\Lambda$,
and so $\tau(z)=\inf\left(\{n>0\colon
f^nz\in\mathcal W^s\}\cup\{\infty\}\right).$



\begin{figure}
\begin{center}
\includegraphics[height=4cm,width=6cm]
{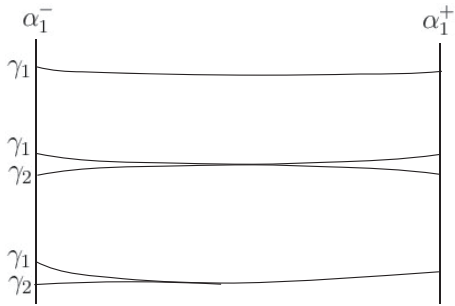}
\caption{The curves in $\Gamma^u$.}
\end{center}
\end{figure}

\begin{lemma}\label{Q}
There exists a collection $\mathcal Q$ of pairwise disjoint subsets of
$\Omega_\infty$ such that:

\begin{itemize}
\item[(a)] $\bigcup_{\omega\in\mathcal Q}\omega=\{z\in\Omega_\infty\colon \tau(z)<\infty\}$;

\item[(b)] $\tau$ is constant on
each $\omega\in\mathcal Q$ not intersecting $\alpha_1^-\cup\alpha^+$ (denote this value by $\tau(\omega)$.
For all other $\omega\in\mathcal Q$, let $\tau(\omega)=2$);

\item[(c)] for each $\omega\in\mathcal Q$ there exists 
$\gamma\in\tilde\Gamma^u$ such that $f^{\tau(\omega)}\overline{\omega}=\gamma\cap\mathcal W^s$.

\item[(d)] for each $\omega\in\mathcal Q$,
$\overline{\omega}\setminus\omega\subset W^s(P)\setminus \{\alpha_1^+\}$.
\end{itemize}
\end{lemma}
We finish the construction of an induced map assuming the conclusions of  Lemma \ref{Q}.
For each $\omega\in\mathcal Q$ in Lemma \ref{Q}, consider the $s$-sublattice of $\Lambda$
defined by the families $\Gamma^u$ and $\{\gamma^s(z)\colon z\in\omega\}$.
Define $\hat S $ to be the collection of
these $s$-sublattices of $\Lambda$.
Let $$B=\{z\in \Omega\colon \tau(z)=\infty\},$$ 
and set
$$\Lambda_B:=\Lambda\cap B.$$
We do not know if $\Lambda_B=\emptyset$.
Since elements of $\hat S$ are $s$-sublattices of $\Lambda$, so is $\Lambda_B$ unless it is an empty set.
One can show that
$\Lambda\setminus\Lambda_B$  is dense in $\Lambda$.

\begin{remark}
{\rm Proposition~\ref{lat} establishes that the first return map of points of $\Lambda$ to itself has good hyperbolic and distortion properties, which allows us to apply the thermodynamical formalism. However, there are non-wandering points whose forward orbits never enter $\Lambda$. The set of all such points is denoted by $B$ and one needs to control its size (in terms of Hausdorff dimension) in order to show that it does not support any equilibrium measures with large entropy.}
\end{remark}

\begin{cor}\label{corQ}
The following holds:
\begin{itemize}

\item[(a)] $\bigcup_{I\in\hat S} I=\Lambda\setminus \Lambda_B$;

\item[(b)] $\tau$ is constant on each $I\in\hat S$ not intersecting $\alpha_1^-\cup\alpha_1^+$ (denote this value by $\tau(I)$.
For all other $I\in\hat S$, let $\tau(I)=2$);

\item[(c)] for each $I\in\hat S$, $f^{\tau(I)}\overline{I}$ is a
$u$-sublattice of $\Lambda$;

\item[(d)]  for each $I\in\hat S$,
$\overline{I}\setminus I\subset W^s(P)\setminus \{P\}$.

\end{itemize}
\end{cor}
\begin{proof}
From Proposition \ref{Q} and Lemma \ref{leaf}(a).
\end{proof}

\begin{figure}
\begin{center}
\includegraphics[height=4cm,width=16cm]{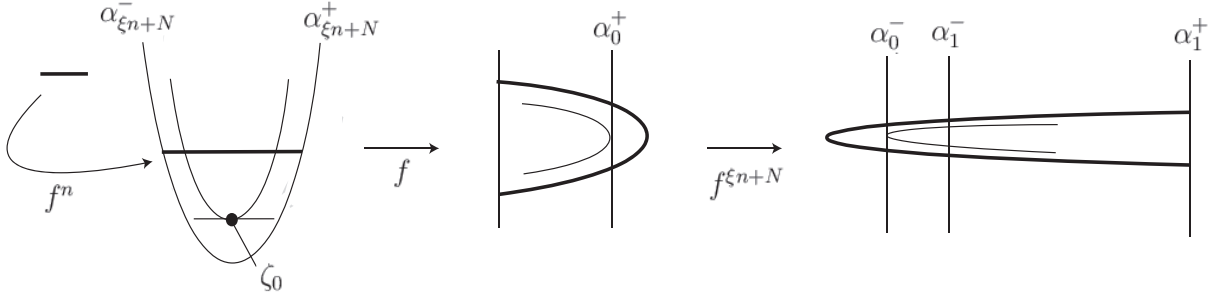}
\caption{The closure of a gap $G$ of $\Omega_\infty$ of order $n$ is mapped by $f^n$ to a $C^2(b)$-curve
connecting the two stable sides of $\Theta_{n}$. Hence $f^{n+1}G$ is folded. $f^{n+i}G\subset \overline{R\setminus\Theta}$ holds
for every $1\leq i\leq\xi n+N$.}
\end{center}
\end{figure}

In order to prove Lemma \ref{Q} we need some preliminary considerations on the geometry 
of $\Omega_\infty$ and $\mathcal W^s$.
\begin{definition}
{\rm The set $\gamma^u(\zeta_0)\setminus\Omega_0$ is called a gap of order $0$.  For each $n\geq1$, any connected component of
$\Omega_{n-1}\setminus\Omega_n$ is called a {\it gap of $\Omega_\infty$ of order
$n$}.

For each $n\geq0$,  by a {\it gap of $\mathcal W^s$
of order $n$} we mean any rectangle bordered by the closure of a gap of
$\Omega_\infty$ of order $n$, a segment in the upper unstable side
of $\Theta$, and two long stable leaves joining their endpoints.}
\end{definition}


\begin{remark}\label{gap remark}
{\rm By construction, the following holds:
\begin{itemize}


\item If $A$ is a component of $\Omega_{n-1}$, then the $f^n$-images of the endpoints of $A$
are contained in $\overline{R\setminus\Theta}$. In particular, any gap of $\Omega_\infty$ of order $n$ intersecting $A$ is strictly contained in $A$. Hence, $\Omega_n\neq\emptyset$ and as a result $\Omega_\infty$ is a nonempty compact set.

\item the closure of each gap of $\Omega_\infty$ of order $n$ is sent by $f^n$ diffeomorphically 
onto a $C^2(b)$-curve connecting the two stable sides of $\Theta_{n}$ (See FIGURE 5).

\item Each gap of $\mathcal W^s$ of order $n\geq0$ is mapped by $f^n$
to a rectangle whose stable sides are in $\alpha_{\xi n+N}^\pm$ and whose unstable sides are $C^2(b)$-curves in $W^u$.
(See FIGURE 6).
\end{itemize}}
\end{remark}

\begin{lemma}\label{gaplem}
Let $G$ be a gap of order $g$. Then for $0\leq i\leq g$,
$f^iG\cap\mathcal W^s=\emptyset$.
\end{lemma}
\begin{proof}
Suppose there exists a point $x\in f^iG\cap\mathcal W^s\neq\emptyset$
for some $0\leq i\leq g$.
Then $f^{g-i}x\in \overline{R\setminus\Theta_{g-i}}$.
On the other hand, 
$f^{g-i}x\in f^gG\subset \Theta_{g}$ and thus 
$f^{g-i}x\cap\overline{R\setminus\Theta_{g}}=\emptyset$,
and $f^{g-i}x\cap\overline{R\setminus\Theta_{g-i}}=\emptyset$, a contradiction.
\end{proof}

\begin{figure}
\begin{center}
\includegraphics[height=7cm,width=10cm]
{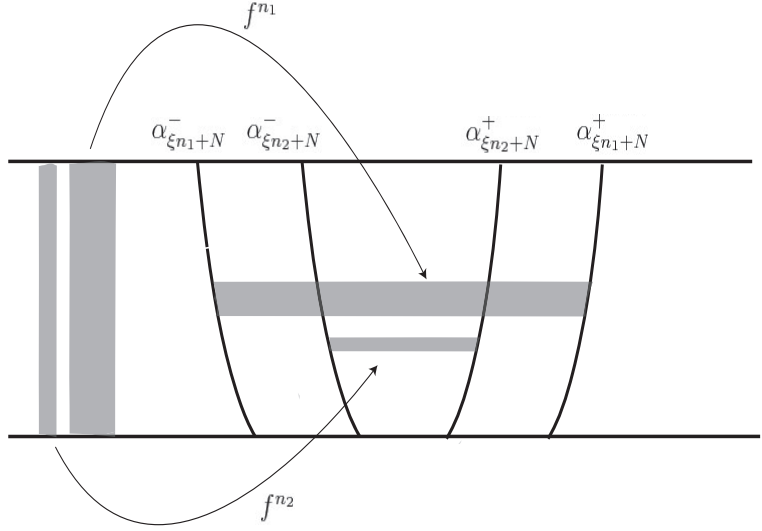}
\caption{Schematic picture of gaps of $\mathcal W^s$. The two shaded vertical rectangles are 
gaps of $\mathcal W^s$ of order $n_1$ and $n_2$, $n_1<n_2$. Each gap of $\mathcal W^s$ of order $n\geq0$ is mapped by $f^n$
to a rectangle whose stable sides are in $\alpha_{\xi n+N}^\pm$ and whose unstable sides are $C^2(b)$-curves in $W^u$.}
\end{center}
\end{figure}

\noindent{\it Proof of Lemma \ref{Q}.}
We construct $\mathcal Q$ by induction.
Consider the partition 
$\mathcal P(\gamma^u(\zeta_0))$ of $\gamma^u(\zeta_0)$ into $C^2(b)$-curves
(See Sect.\ref{family} for the definition of this partition),
 and set $\mathcal P_0=\mathcal P(\gamma^u(\zeta_0))|\Omega_0$. 
For each $\omega\in\mathcal P_0$ let $2\leq \tau(\omega)\leq N$ denote
the unique 
integer such that $f^{\tau(\omega)}\omega\subset\tilde\Gamma^u$. 
There are exactly two elements of $\mathcal P_0$ for which $\tau(\omega)=2$.
If $\tau(\omega)=2$, then
we let
$f^{-\tau(\omega)}\left(f^{\tau(\omega)}\omega\cap\mathcal W^s \right)\in\mathcal
Q.$ If $\tau(\omega)>2$, then 
we let
$f^{-\tau(\omega)}\left(f^{\tau(\omega)}\omega\cap\mathcal W^s\setminus (\alpha_1^-\cup\alpha_1^+)\right)\in\mathcal
Q.$ We remove these pre-images of $\alpha_1^-\cup\alpha_1^+$ to ensure that elements of $\mathcal Q$ are pairwise disjoint.
The next sublemma justifies this construction.
\begin{sublemma}\label{match}
For each $\omega\in\mathcal P_0$,
$f^{-\tau(\omega)}\left(f^{\tau(\omega)}\omega\cap\mathcal W^s\right)\subset
\Omega_\infty$.
\end{sublemma}
\begin{proof}
Recall that $\Omega_\infty$ is the set of points which do not approach the critical region too close and too soon.
In the proof below we use the fact that if a point in a $\gamma^s$-curve is in $\Omega_\infty
$ then the curve satisfies the same property. 

Let $z\in f^{-\tau(\omega)}\left(f^{\tau(\omega)}\omega\cap\mathcal W^s\right)$.
Since
$f^{-\tau(\omega)}\left(f^{\tau(\omega)}\omega\cap\mathcal W^s\right)\subset \Omega_{\tau(\omega)}$ by the definition
of $\tau(\omega)$,
to conclude $z\in\Omega_\infty$
it suffices to show $z\in\Omega_{n+\tau(\omega)}$ for every $n>0$. 

There exists $y\in\Omega_\infty$ such that
 $f^{\tau(\omega)}z\in\gamma^s(y)$.
Then $f^ny\in \overline{R\setminus\Theta_{ n}}$ by \eqref{sr}. 
Suppose that $f^{n+\tau(\omega)}z\notin \overline{R\setminus\Theta_{ n}}$, namely $f^{n+\tau(\omega)}z\in\Theta_{n}$. Then
$f^n\gamma^s(y)$ would intersect the stable side of $\Theta_{n}$. 
Also, since $y\in\Omega$ the forward iterates of $\gamma^s(y)$ are contained in a bounded region, and so $f^n\gamma^s(y)\subset R$.
Because of the contraction along $\gamma^s(y)$ and since the stable sides of $\Theta_{n}$ are contained in $W^s(P)$, this would imply $f^n\gamma^s(y)\subset W^s(P)$. 
However, it is not possible to connect points in $\overline{R\setminus\Theta_{n}}$
and $\Theta_{n}$ by curves in $W^s(P)\cap R$.
Hence we obtain $f^{n+\tau(\omega)}z\notin \overline{R\setminus\Theta_{ n}}$ 
and so $z\in\Omega_{n+\tau(\omega)}$.
\end{proof}

For the next step of the induction, consider a gap $G$ of $\mathcal W^s$
of order $g$ and let $\gamma\subset\omega\in\mathcal Q$ be such that
$f^{\tau(\omega)}\gamma$ is a $C^2(b)$ curve connecting the two stable sides of $G$.
Let $\omega'\subset\gamma$ be the preimage under $f^{\tau(\omega)+g}$ of an element
of the partition $\mathcal P(f^{\tau(\omega)+g}\gamma)$ such that $\omega'$ 
contains points of $\Omega_\infty$. 
Then  $f^{m+\tau(\omega)+g}\omega'\in \tilde\Gamma^u$,
where $m=
\tau(f^{\tau(\omega)+g}\omega')$.
Let $\tilde\omega:=f^{-m-\tau(\omega)-g}\left(f^{m+\tau(\omega)+g}\omega'\cap\mathcal
W^s\setminus(\alpha_1^-\cup\alpha_1^+)\right)$.
The next sublemma justifies this construction.

\begin{sublemma}\label{match2}
$\tilde\omega\subset\Omega_\infty $.
\end{sublemma}
\begin{proof}
By the definition of $\tau(\omega)$, we have
$f^i\omega'\subset \overline{R\setminus\Theta}$ for $0< i<\tau(\omega)$.
We have $f^{\tau(\omega)}\omega'\subset G$, and so by the definition of the gap $G$ of order $g$, 
$f^j(f^{\tau(\omega)}\omega')\subset \overline{R\setminus\Theta_{ j}}$
for $0\leq j \leq g-1$,  and 
$f^{g+\tau(\omega)}\omega'\subset\Theta_{ g}$.
Since
$\Theta_{\tau(\omega)+j}\subset \Theta_{ j}$,
we in fact get
$f^{j}(f^{\tau(\omega)}\omega')\subset\overline{R\setminus\Theta_{j+\tau(\omega)}}$
for $0\leq j\leq g-1$.
However, since $\omega'\cap\Omega_\infty\neq\emptyset$
we have $f^{\tau(\omega)+g}\omega'\subset \Theta_{ g}\setminus
\Theta_{g+\tau(\omega)}$,
and thus
 $\omega'\subset\Omega_{g+\tau(\omega)}.$

By the definition of $m$,
we have
 $f^{n}(f^{\tau(\omega)+g}\omega')\subset \overline{R\setminus \Theta}$ for every
$0<n<m$, and 
 so $\tilde\omega\subset\omega'\subset \Omega_{m+g+\tau(\omega)-1}$.
 In addition,
 $f^{m+g+\tau(\omega)}\tilde\omega\subset\mathcal W^s\subset \overline{R\setminus\Theta_0}$ 
  imply 
$\tilde\omega\subset\Omega_{m+g+\tau(\omega)}$. 
The argument of Lemma \ref{match} shows 
$f^{n}(f^{m+g+\tau(\omega)}\tilde\omega)\subset \overline{R\setminus\Theta_{ n}}$
for every $n>0$,
and so
$\tilde\omega\subset\Omega_{n+m+g+\tau(\omega)}$.
\end{proof}

By virtue of Lemma \ref{gaplem}, this construction can be repeated
indefinitely. We can finish the construction of
$\mathcal Q$ by induction.
Statements (a)-(c) of Lemma \ref{Q} are immediate consequences of the construction. 

We prove statement (d). Let $\omega\in \mathcal Q$.
By construction, $\overline{f^{\tau(\omega)}\omega}\setminus f^{\tau(\omega)}\omega\subset \alpha_1^-\bigcup\alpha_1^+$.
Hence, $\overline{\omega}\setminus \omega\subset W^s(P)$.
By construction, if $\tau(\omega)>2$ then
$\alpha_1^+\cap\overline{\omega}=\emptyset$.
If $\tau(\omega)=2$ then either $\alpha_1^+\cap\overline{\omega}=\emptyset$, or else $\alpha_1^+\cap \omega\neq\emptyset$.
In both cases $\alpha_1^+\notin\overline{\omega}\setminus \omega.$
Hence $\overline{\omega}\setminus \omega\subset W^s(P)\setminus\alpha_1^+$.
\qed

\subsection{Construction of an inducing scheme from the induced map}\label{symbolcode}
From $(\hat S, \tau)$ in Sect.\ref{induced}
we now construct an inducing scheme $(S,\tau)$ as in Proposition \ref{lat}.

Define an induced map $F\colon\bigcup_{I\in\hat S} I\to\Lambda$ by
$Fz=f^{\tau(I)}z$, where $I$ is the element of $\hat S$ containing $z$. 
Following the terminology and notation in Definition \ref{scheme}, define the {\it inducing domain} $X$ by $$X=\bigcap_{n=-\infty}^\infty
F^{-n}\left(\bigcup_{I\in\hat S} I\right).$$ By Lemma \ref{X} below, $X\neq\emptyset$.
Define the collection $S$ of {\it basic elements} by
$$S=\{I\cap X\colon I\in \hat S\ \text{and}\ I\cap X\neq\emptyset\}.$$
By definition, elements of $S$ are pairwise disjoint and $X=\bigcup_{J\in S}J$ holds.
Set $$X^*=\bigcup_{J\in S} \overline{J}.$$
Define $\tau\colon S\to\mathbb N$
by $\tau(J)=\tau(I)$, where $I$ is the element of $\hat S$ containing $J$.
\medskip

\begin{lemma}\label{X}
There are families ${\Gamma^u}'\subset\Gamma^u$, ${\Gamma^s}'\subset\Gamma^s$
such that $X$ is a lattice defined by ${\Gamma^u}'$ and ${\Gamma^s}'$.
\end{lemma}
\begin{proof}
Write $X=X^-\cap X^+$ where
$X^-=\bigcap_{n=1}^{\infty}
F^{n}\left(\bigcup_{I\in\hat S} I\right)$ and $X^+=\bigcap_{n=0}^{\infty}
F^{-n}\left(\bigcup_{I\in\hat S} I\right)$.  
Since there is an element of $\hat S$ containing the fixed saddle $P$, $X^+$ contains $P$ 
and so is not an empty set. By construction, $X^+$ is written as a union of curves in $\Gamma^s$.

For each $n\geq1$ let $\Gamma^u_n$ denote the defining family of $u$-curves of
the lattice $F^{n}\left(\bigcup_{I\in\hat S} I\right)$.
Then $\bigcup_{\gamma^u\in\Gamma^u_n}\gamma^u$ is a closed set, strictly decreasing in $n$.
Hence $X^+$ is the union of $u$-curves in $\bigcap_{n=1}^\infty\Gamma^u_n$.
\end{proof}

 We now prove (P1)-(P4) from Proposition \ref{lat},
and then verify (A1)-(A3) from Definition \ref{scheme}. 
 Proofs of (P5) and (P6) are deferred to Sect.\ref{opt}.
 \medskip

\noindent{\it Proof of (P1).}
Let $p\colon \Lambda\to\gamma^u(\zeta_0)$ denote the holonomy map along $\gamma^s$-curves, i.e.,
$p(x)=\gamma^u(\zeta_0)\cap\gamma^s(x)$.
For subsets $A$, $B$ of $\Lambda$ we use the notation $A=_sB$ if $pA=pB$.
The meaning of $A\subset_s B$ is analogous.

 Let $J\in S$. By Lemma \ref{X}, $J$ is a lattice defined by ${\Gamma^u}'$ and
a subset of ${\Gamma^s}'$. 
By construction,
for each  $\gamma^u\in{\Gamma^u}'$,
$f^{\tau(J)}(\gamma^u\cap J)\in\Gamma^u.$
Hence, it suffices to show \begin{equation}\label{P1eq}
f^{\tau(J)}J\subset_sX\text{ and }f^{\tau(J)}\overline{J}=_sX^*.\end{equation}

 Let  $I$ denote the element of $\hat S$
containing $J$. 
Then $J=_sI\setminus
\bigcup_{n\geq0}F^{-n}\Lambda_B$. Hence
$f^{\tau(J)}J=_sf^{\tau(J)}I\setminus
\bigcup_{n\geq0}F^{-n+1}\Lambda_B=f^{\tau(J)}I\setminus \bigcup_{n\geq0}F^{-n}\Lambda_B
\subset\Lambda\setminus \bigcup_{n\geq0}F^{-n}\Lambda_B=_s X
$. Hence the first item in \eqref{P1eq} holds. The second one follows from 
$f^{\tau(J)}J=_sf^{\tau(J)}I\setminus \bigcup_{n\geq0}F^{-n}\Lambda_B
$.
\medskip

\noindent{\it Proofs of (P2)-(P4).}
(P2) follows from the backward contraction on the leaves in $\tilde\Gamma^u$ 
(see \cite[Lemma 4.2]{SenTak11}) and the fact that any leaf in $\Gamma^u$ 
is a $C^1$-limit of leaves in $\tilde\Gamma^u$.
Since $f^{\tau(J)}$ is a composition of first return maps to $\Theta$, 
(P3) and (P4)(a) follow from the estimates in Proposition \ref{markov}.
(P4)(b) follows from Lemma \ref{leaf}(b).
\qed
\medskip

\noindent{\it Verification of (A1).}
Included in (P1).
\medskip

\noindent{\it Verification of (A2).}
Define a coding map $h\colon S^\mathbb Z\to X^*$ by 
\begin{equation}\label{hdefi}h(\underline{a}):=
\overline{J_{0}}\cap\left(\bigcap_{n\geq1} f^{-\tau(J_{0})}\circ \cdots\circ f^{-\tau(J_{n-1})}(\overline{J_n})\right)
\cap \left(\bigcap_{n\geq1} f^{\tau(J_{-1})}\circ \cdots\circ f^{\tau(J_{-n})}(\overline{J_{-n}})\right).
\end{equation}
To verify (A2) we show that $h$ is well-defined, and induces
a measurable bijection between $S^{\mathbb Z}\setminus h^{-1}(X^*\setminus X)$ and $X$.
\medskip


\noindent{\it Well-definedness.} 
Since the set on the right-hand-side of \eqref{hdefi} is an intersection of a nested closed sets, it is nonempty. To show that it is a singleton, for each $n\geq1$ denote by $Q_n^s(\underline{a})\subset\Lambda$ the rectangle spanned by the closed $s$-sublattice:
$$Q_n^s(\underline{a}):=
\overline{J_{0}}\cap\left(\bigcap_{k=1}^n f^{-\tau(J_{0})}\circ \cdots\circ f^{-\tau(J_{k-1})}(\overline{J_k})\right),
$$
and define $Q_n^u(\underline{a})$ similarly:
$$Q_n^u(\underline{a}):=
\overline{J_{0}}\cap\left(\bigcap_{k=1}^n f^{\tau(J_{-1})}\circ \cdots\circ f^{\tau(J_{-k})}(\overline{J_k})\right).
$$
We have $$h(\underline{a})\subset \left(\bigcap_{n\geq1}Q_n^s(\underline{a})\right)
\cap\left(\bigcap_{n\geq1}Q_n^u(\underline{a})\right).$$
By (P1), 
$\{Q_n^s(\underline{a})\}$ is a nested sequence of nonempty closed sets.
Hence, $\bigcap_{n\geq1}Q_n^s(\underline{a})\neq\emptyset$.
By (P3) and the fact that $\Lambda$ is closed, 
the boundary curves of $Q_n^s(\underline{a})$ in $\Gamma^s$ converge, in the $C^0$ topology, to a curve in $\Gamma^s$.
Hence, $\bigcap_{n\geq1}Q_n^s(\underline{a})\in\Gamma^s$.
In the same way we have $\bigcap_{n\geq1}Q_n^u(\underline{a})\in\Gamma^u$.
Since each curve in $\Gamma^s$ intersects each curve in $\Gamma^u$ exactly at one point,
it follows that the set on the right-hand-side of \eqref{hdefi} is a singleton. So, 
 $h$ is well-defined.
\medskip

\noindent{\it Measurable bijectivity.} 
Let $\underline{a}$, $\underline{a'}\in S^{\mathbb Z}\setminus h^{-1}(X^*\setminus X)$
and suppose $h(\underline{a})=h(\underline{a}')=x$. 
We have $x\in X$, and (P1) gives $F^nx\in X$ for every $n\geq1$.
Since elements of $S$ are pairwise disjoint, $a_n=a_n'$ holds for every $n\geq 0$.
In addition, the next lemma implies $a_n=a_n'$ for every $n< 0$.

\begin{lemma}\label{bdisjoint}
For $J$, $J'\in S$ distinct either $f^{\tau(J)}\overline{J}\cap f^{\tau(J')}\overline{J'}=\emptyset$ or $f^{\tau(J)}\overline{J}\cap f^{\tau(J')}\overline{J'}\subset X^*\setminus X$.
\end{lemma}
\begin{proof}
It is obvious that if $\tau(J)=\tau(J')$ then $f^{\tau(J)}\overline{J}\cap f^{\tau(J')}\overline{J'}=\emptyset$.
Suppose $\tau(J)>\tau(J')$ and
 $f^{\tau(J)}\overline{J}\cap f^{\tau(J')}\overline{J'}\neq\emptyset$.
Then $f^{\tau(J)-\tau(J')}\overline{J}\cap\overline{J'}
\neq\emptyset$ and since $\tau(J)$ is the first entry time of $J$ to $\Lambda$ and $J'\subset\Lambda$, then $f^{\tau(J)-\tau(J')}\overline{J}\cap\overline{J'}\subset \overline{J'}\setminus J'$. The intersection thus belongs to $X^*\setminus X$.
\end{proof}

Consequently, $\tilde h:= h|_{S^{\mathbb Z}\setminus h^{-1}(X^*\setminus X)}$
is injective. Since $h$ is onto $X$, $\tilde h$ is onto $X$ as well.
From the uniform hyperbolicity of $F$ and the fact that
the cylinder sets form a base of the topology in $S^{\mathbb Z}$,
$\tilde h$ is continuous and maps open sets to Borel sets.
By  \cite[Claim 3.3]{SenTak11}, $\tilde h^{-1}$ is measurable.
\medskip

\noindent{\it Verification of (A3).}
From Corollary \ref{Q}(d) and the construction, 
 $\overline{J}\setminus J\subset (W^s(P)\setminus\{P\})
\cup\bigcup_{n\geq0}F^{-n}\Lambda_B$ holds for each $J\in S$.
By the next lemma,
no $f$-invariant probability measure gives positive weight to $X^*\setminus X$.
\begin{lemma}\label{setB}
$\bigcup_{n\geq0}F^{-n}\Lambda_B$ is a $s$-sublattice of $\Lambda$. In addition, for any $\mu\in\mathcal M(f)$ one has that
$\mu(\bigcup_{n\geq0}F^{-n}\Lambda_B)=0$.
\end{lemma}
\begin{proof}
As in the previous remark, $\Lambda_B$ is an $s$-sublattice of $\Lambda$.
Since $F$ maps each $\gamma^s$-curve into a subset of an $\gamma^s$-curve, 
$F^{-n}\Lambda_B$ is an $s$-sublattice of $\Lambda$
for every $n>0$. Hence the first statement holds.
As for the second one, since $f^{-n}\Lambda_B$ $(n=0,1,\ldots)$ are pairwise disjoint,
$\mu(\Lambda_B)=0$. Hence $\mu(f^{-n}\Lambda_B)=0$ and so $\mu(\bigcup_{n\geq0}F^{-n}\Lambda_B)=0$.
\end{proof} 
 
If there is a $\sigma$-invariant probability measure $\nu$ which gives positive weight to $h^{-1}(X^*\setminus X)$,
then $h_*\nu$ gives positive weight to $X^*\setminus X$, and so does
$\mathcal L(h_*\nu)\in\mathcal M(f)$. This is a contradiction.

\subsection{Comments on proofs of (P5) (P6)}\label{opt}
{To prove (P5) and (P6) we use the fact that the inducing scheme $(S,\tau)$ is built over the first entry time to $\Lambda$. 
%
To prove (P6) it suffices to show that any ergodic measure with entropy $\geq2\varepsilon$ gives positive weight to $X$. This amounts to showing that the first entry time to $\Lambda$ is finite, with positive probability for these measures.

The rest of this section is organized as follows. Sect.\ref{exceptional} and Sect.\ref{ncontrol} are devoted to estimates of the Hausdorff dimension of sets of points not returning to $\Lambda$. (P5) is proved in Sect.\ref{small} with the estimate obtained in Sect.\ref{exceptional}. (P6) is proved  in Sect.\ref{pos}.



\begin{figure}
\begin{center}
\includegraphics[height=5cm,width=10cm]
{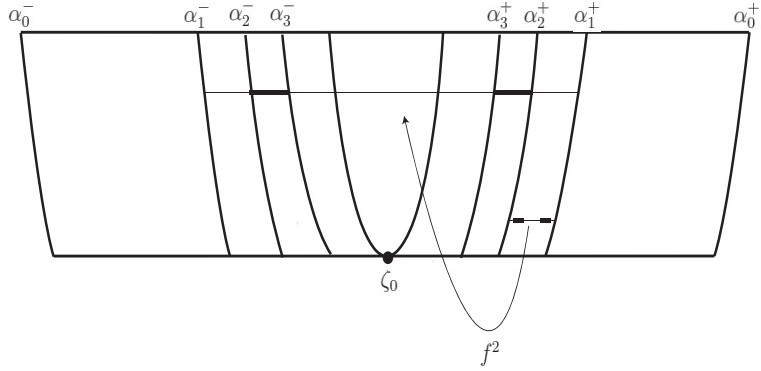}
\caption{
$f^{k_1+\cdots+k_{i-1}}$-images of $\omega_i\in\mathcal Q(k_1,\ldots,k_{i-1},2)$ and the two elements of $\mathcal Q(\omega_i,3)$,
and their $f^{2}$-images.}
\end{center}
\end{figure}

\subsection{Hausdorff dimension of the set of points in $\Lambda$ not returning to $\Lambda$}\label{exceptional}
By Lemma \ref{leaf}, the tangent directions of the curves in $\Gamma^s$ are Lipschitz continuous, and 
so the holonomy map between two curves in $\tilde\Gamma^u$ along $\gamma^s$-curves is Lipschitz continuous. Hence, for any $s$-sublattice $\Lambda'$ of $\Lambda$ the Hausdorff dimension of $\gamma^u\cap \Lambda'$,
denoted by $d^u(\Lambda')$,  is independent of the choice of $\gamma^u\in\tilde\Gamma^u$.

\begin{lemma}\label{bad}
$d^u(\Lambda_B)\leq\varepsilon.$
\end{lemma}
\begin{proof} 
Since $d^u(\Lambda_B)=\dim^u_H\left(\Omega_\infty\cap B\right)$
it suffices to show $\dim^u_H\left(\Omega_\infty\cap B\right)\leq\varepsilon.$
Given $g>1$ and a $g$-string $(k_1,\ldots,k_g)$ of positive integers,
we define collections
$\mathcal Q(k_1,k_2,\ldots,k_i)$ $(i=1,2,\ldots,g)$
of pairwise disjoint curves in $\gamma^u(\zeta_0)$ inductively as follows. Let
$$\mathcal Q(k_1)= \{\omega_{1}\subset\gamma^u(\zeta_0)\colon
f^{k_1}\omega_1\in\tilde\Gamma^u\}.$$ 
Given $\mathcal Q(k_1,\ldots,k_{i})$,
for
each $\omega_{i}\in\mathcal Q(k_1,\ldots,k_{i})$ let
$$\mathcal Q(\omega_{i},k_{i+1})= \{\omega_{i+1}\subset\omega_{i}\colon
f^{k_1+\cdots+k_{i}+k_{i+1}}\omega_{i+1}\in\tilde\Gamma^u\},$$
and define inductively
$$\mathcal Q(k_1,\ldots,k_{i+1})=
\bigcup_{\omega_{i}\in \mathcal Q(k_1,\ldots,k_{i}) }
\mathcal Q(\omega_{i},k_{i+1}).$$
(See FIGURE 7).

For $n>0$ let
$$\mathcal Q_n(k_1,\ldots,k_i)=\{\omega_i\in \mathcal Q(k_1,\ldots,k_i)\colon 
\sup\{\tau(z)\colon z\in\omega_i\}\geq n\},$$
and
$$\mathcal Q_n(\omega_{i},k_{i+1})=\{\omega_{i+1}\in\mathcal Q(\omega_{i},k_{i+1})
\colon \sup\{\tau(z)\colon z\in\omega_{i+1}\}\geq n\}.$$
Let
$\omega_0=\gamma^u(\zeta_0)$ and $\mathcal Q_n(\omega_0,k_1)=\mathcal Q_n(k_1).$
\begin{sublemma}\label{card}
If $N>2(1+\xi)$, then for every $n>6N$ and 
for any $z\in \Omega_\infty$ with $\tau(z)\geq n$ there exist an integer $1\leq s\leq n/N$, 
and for each $i=1,\ldots,s$ an integer $k_i\geq N$ 
and a curve
$\omega_i\in\mathcal Q_n(k_1,\ldots,k_i)$ such that:
\begin{itemize}
\item[(a)] $\frac{n}{3\xi}\leq k_1+\cdots+k_{s}\leq n $;

\item[(b)] $z\in\omega_{s}\subset\cdots\subset\omega_1;$

\item[(c)]
$\ell(\omega_{s})\leq
C\sigma_1^{-(k_1+\cdots+k_{s})}$;

\item[(d)] for each $i=0,\ldots,s-1$, $\# \mathcal Q_n(\omega_{i},k_{i+1})<
2^{\frac{k_{i+1}}{\xi}}.$
\end{itemize}
\end{sublemma}

\begin{proof}
For $z\in\Omega_\infty$
define a sequence $0=:t_0<t_1<\cdots$ of return times to
$\Theta$ inductively as follows:
given $t_{i}$ such that $f^{t_{i}}z$ is in the gap of $\mathcal
W^s$ of order $g_i$, define 
$$t_{i+1}=\min\{t>t_i+g_i\colon f^tz
\in\Theta\}.$$
Note that $\{t_i\}$ are not the only return times of the orbit of $z$ to $\Theta$. 
Since $f^{t_{i}}z$ is in the gap of order $g_i$, one has $t_{i+1}\geq t_i+g_i+\xi g_i+N$ and $g_i\geq0$ we have
\begin{equation}\label{d1}t_{i+1}-t_{i}\geq N.\end{equation}
Define $s:=\max\{i\colon t_i< n\}+1$. (\ref{d1}) gives
 $s\leq n/N$. 
\begin{claim}\label{q}
$t_{s-1}+g_{s-1}\geq\frac{n}{3\xi}$.
\end{claim}
\begin{proof}
Suppose the contrary. Then $\xi(t_{s-1}+g_{s-1})+N<[n/2]$.
Since $f^{t_{s-1}}z$ is in a gap of order $g_{s-1}$ one has $f^{t_{s-1}+g_{s-1}}z\in\Theta_{g_{s-1}}\subset\Theta.$
Set $r=t_s-t_{s-1}-g_{s-1}$.
Since $r$ is the first return time of $f^{t_{s-1}+g_{s-1}}z$ to $\Theta$,
$f^{t_{s-1}+g_{s-1}}z$ belongs to a $C^2(b)$-curve connecting either $\alpha_{r}^+$ and $\alpha_{r+1}^+$, or $\alpha_{r}^-$ and $\alpha_{r+1}^-$. 
Since $f^{t_{s-1}+g_{s-1}}z\notin\Theta_{ t_{s-1}+g_{s-1} }$ by the assumption $z\in \Omega_\infty$ (see \eqref{sr}), one has  
$r<\xi(t_{s-1}+g_{s-1})+N<[n/2]$ and $t_{s}<\frac{n}{3\xi}+[n/2]<n$, a contradiction to $t_s\geq n$. 
\end{proof}

For each $i=1,\ldots,s$,
set $k_{i}=t_{i}-t_{i-1}$ and let $\omega_i$ denote the element of $\mathcal Q_n(k_1,\ldots,k_i)$
containing $z$. 
\eqref{q} and the definition of $s$ give $t_{s-1}+g_{s-1}<t_s\leq n$.
Since $k_1+\cdots+k_s=t_s$, (a) holds.
(b) is straightforward. (c) follows from Lemma \ref{markov}. 

In order to prove (d),
observe that
for each $\omega\in \mathcal Q_n(\omega_{i},k_{i+1})$
 there exists a gap $G$ of $\mathcal W^s$ such that
 $f^{k_1+\cdots+k_{i}}\omega\subset G.$
In addition, for any gap $G$ of $\mathcal W^s$ 
we have \begin{equation}\label{num}
\#\{\omega\in  \mathcal
Q_n(\omega_{i},k_{i+1})\colon f^{k_1+\cdots+k_{i}}\omega\subset G \}=0\
\text{or}\ =2,\end{equation}
which is because gaps are not folded up to their order, 
and
$f^jG\cap\Theta=\emptyset$
for $g_i<j<t_{i+1}$
by the definition of $t_{i+1}$.
There is therefore at most one fold at time $t_i+g_i$. 

Let
$g_0$ denote the maximal order of the gap of $\mathcal W^s$ which 
contains $f^{k_1+\cdots+k_{i}}$-images of elements
of $\mathcal Q_n(\omega_{i},k_{i+1}).$ 
Then $g_0+\xi 
g_0+N\leq k_{i+1}$ holds. From \eqref{num} and the fact that the number of gaps
of order $g$ is $\leq 2^g$ we obtain
$\#\mathcal Q_n(\omega_{i},k_{i+1})\leq 2\sum_{i=1}^{g_0}2^i<2^{\frac{k_{i+1}-N}{1+\xi}+2}
<2^{\frac{k_{i+1}}{\xi}}.$
The last inequality holds provided $N>2(1+\xi)$.
\end{proof}
Returning to the proof of Lemma \ref{bad}, we have
$$\Omega_\infty\cap B\subset\{z\in
\Omega_\infty\colon \tau(z)\geq n\}\subset\bigcup_{s=1}^{\left[\frac{n}{N}\right]}
\bigcup_{l=\left[\frac{n}{3\xi}\right]}^n\bigcup_{k_1+\cdots+k_{s}=l}
\bigcup_{\omega_{s}\in \mathcal Q_n(k_1,\ldots,k_{s})}\omega_{s}.$$
By Sublemma \ref{card}(a)(c), 
the lengths of the curves $\omega_{s}$ in the union of the right-hand-side
are exponentially small in $n$. 
We show that $\sum_{\text{all relevant $\omega_s$}}\ell(\omega_s)^{\varepsilon}$ is finite for all $n$.

Observe that
\begin{equation}\label{eqbad}
\sum_{\omega_{i+1}\in\mathcal Q_n(k_1,\ldots,k_{i+1})}
\ell(\omega_{i+1})^{\varepsilon}=\sum_{\omega_{i}\in
\mathcal Q_n(k_1,\ldots,k_{i})}
\ell(\omega_{i})^\varepsilon\sum_{\omega_{i+1}\in\mathcal
Q_n(\omega_{i},k_{i+1}) } \frac{\ell(\omega_{i+1})^\varepsilon}
{\ell(\omega_{i})^\varepsilon}.\end{equation}
On the second sum of the fractions, 
since $\ell(f^{k_1+\cdots+k_{i+1}}\omega_{i+1} )<2$ 
for each $\omega_{i+1}\in\mathcal Q_n(\omega_{i},k_{i+1})$
and
$\|D_xf^{k_{i+1}}|E^u\|\geq\sigma_1^{k_{i+1}}$ for all $x\in
f^{k_1+\cdots+k_{i}}\omega_{i+1}$, we have
$\ell(f^{k_1+\cdots+k_{i}}\omega_{i+1})\leq (1/2)\sigma_1^{-k_{i+1}}.$
From this and the bounded distortion in Lemma \ref{markov},
\begin{equation}\label{d4}\frac{\ell(\omega_{i+1})}{\ell(\omega_i)}\leq
C\cdot\frac{\ell(f^{k_1+\cdots+k_{i}}\omega_{i+1})}{\ell(f^{k_1+\cdots+k_{i}}\omega_i)}\leq
C\sigma_1^{-k_{i+1}}.\end{equation} 
Using \eqref{d4} and 
Sublemma \ref{card}(d),
$$\sum_{\omega_{i+1}\in\mathcal Q_n(\omega_{i},k_{i+1})
}\frac{\ell(\omega_{i+1})^\varepsilon}{\ell(\omega_{i})^\varepsilon} \leq
\#\mathcal Q_n(\omega_{i},k_{i+1})C^\varepsilon\sigma_1^{-\varepsilon
k_{i+1}}\leq C^{\varepsilon}\sigma_1^{-\frac{\varepsilon}{2} k_{i+1}}.$$ 
Plugging this
into the right-hand-side of \eqref{eqbad} we get
\begin{equation}\label{d-1}
\sum_{\omega_i\in\mathcal Q_n(k_1,\ldots,k_{i+1})}\ell(\omega_{i+1})^\varepsilon
\leq C^\varepsilon\sigma_1^{-\varepsilon k_{i+1}}
\sum_{\omega_{i}\in\mathcal Q_n(k_1,\ldots,k_{i})}\ell(\omega_{i})^\varepsilon
.\end{equation} 
The same arguments as above applied to any
 $\omega_1\in\mathcal Q_n(k_1)$ yield
\begin{equation}\label{d-2}
\sum_{\omega_1\in\mathcal Q_n(k_1)}\ell(\omega_{1})^\varepsilon
\leq C^{\varepsilon}\#\mathcal Q_n(k_{1})\sigma_1^{-\varepsilon
k_{1}}\leq C^\varepsilon\sigma_1^{-\frac{\varepsilon}{2} k_{1}}.\end{equation}
Using (\ref{d-1}) inductively and (\ref{d-2})
yields
\begin{equation*}
\sum_{\omega_s\in\mathcal Q_n(k_1,\ldots,k_s)}\ell(\omega_s)^\varepsilon
\leq C^{\varepsilon s
}\sigma_1^{-\frac{\varepsilon}{2}(k_1+\cdots+k_s)}.\end{equation*} 
Hence
\begin{align*}
\sum_{l=\left[\frac{n}{3\xi}\right]}^n\sum_{k_1+\cdots+k_s=l}\sum_{\omega_s\in\mathcal
Q_n(k_1,\ldots,k_s)}\ell(\omega_s)^\varepsilon&\leq C^{\varepsilon s}
\sum_{l=\left[\frac{n}{3\xi}\right]}^n
\sigma_1^{-\frac{\varepsilon l}{2}}\#\left\{
(k_1,\ldots,k_s)\colon \sum_{i=1}^s k_i=l\right\}.
\end{align*}
To estimate the right-hand side we use the following from Stirling's formula for factorials (see e.g. \cite{Fel68}): for sufficiently small $\chi>0$ there exist $c(\chi)>0$ with $c(\chi)\to0$ as $\chi\to0$ such that for any two positive integers $p$, $q$ with $q/p\leq\chi$ one has
$\left(\begin{smallmatrix}p+q\\q\end{smallmatrix}\right)\leq e^{c(\chi)p}$.
Since $k_i\geq N$ and $s/l\leq 1/N$, one has that $e^{c(1/N)}\le \sigma_1^{\frac{\varepsilon}{3}}$ for sufficiently large $N$ and thus
$$\#\left\{
(k_1,\ldots,k_s)\colon \sum_{i=1}^s k_i=l\right\}\leq \begin{pmatrix}l+s\\s\end{pmatrix}
\leq \sigma_1^{\frac{\varepsilon}{3}l}.$$

Then
\begin{align*}
\sum_{l=\left[\frac{n}{3\xi}\right]}^n\sum_{k_1+\cdots+k_s=l}\sum_{\omega_s\in\mathcal
Q_n(k_1,\ldots,k_s)}\ell(\omega_s)^\varepsilon
&\leq
C^{\varepsilon s}\sum_{l=\left[\frac{n}{3\xi}\right]}^n\sigma_1^{-\frac{\varepsilon}{6} l}.
\end{align*} Hence
\[\sum_{s=1}^{\left[\frac{n}{N}\right]}\sum_{l=\left[\frac{n}{3\xi}\right]}^n\sum_{k_1+\cdots+k_s=l}\sum_{\omega_s\in\mathcal
Q_n(k_1,\ldots,k_s)}\ell(\omega_s)^\varepsilon\leq
\frac{C^{\frac{\varepsilon n}{N}}}{C^\varepsilon-1}\sum_{l=\left[\frac{n}{3\xi}\right]}^n
\sigma_1^{-\frac{\varepsilon}{6} l }.\] 
One can choose $N\geq N_0$ large enough at the expense of reducing $b>0$ so that
the expression on the right-hand-side decays exponentially with $n$.
Consequently the
Hausdorff $\varepsilon$-measure of $\Omega_\infty\cap B$ is zero.\end{proof}

\subsection{Proof of (P5)}\label{small}
We are in position to prove (P5) in Proposition \ref{lat}.
\medskip

\noindent{\it Proof of (P5).}
For each $J\in S$ with $\tau(J)=n$, let $\omega_{J}$ denote the
unstable side of the rectangle $Q_{J}$ spanned by $J$ which is
contained in $\gamma^u(\zeta_0)$. Observe that there exists $1\leq s\leq n/N$
and an $s$-string $(k_1,\ldots,k_s)$ of positive integers such that 
$k_1+\cdots+k_s=n$ and  $\omega_{J}\in\mathcal Q_n(k_1,\ldots,k_s)$. 
For two distinct $J_1,J_2\in S$ with $\tau(J_1)=\tau(J_2)=n$, one has $\omega_{J_1}
\cap\omega_{J_2}=\emptyset.$ Therefore,
$$S(n)
\leq\sum_{s=1}^{[\frac{n}{N}]} \sum_{k_1+\cdots+k_s=n}\#\mathcal
Q_n(k_1,\ldots,k_s).$$
Using Sublemma \ref{card}(d) repeatedly and $\#\mathcal Q_n(k_1)\leq 2$,
\begin{align*}\#\mathcal
Q_n(k_1,\ldots,k_s)&=  \sum_{\omega\in\mathcal
Q_n(k_1,\ldots,k_{s-1})}\#\mathcal
Q(\omega,k_{s})\\&\leq\#\mathcal
Q_n(k_1,\ldots,k_{s-1})\max_{\omega\in\mathcal
Q_n(k_1,\ldots,k_{s-1})}\#\mathcal
Q(\omega,k_{s})\\
&\leq\#\mathcal
Q_n(k_1,\ldots,k_{s-1})\cdot 2^{\frac{k_s}{\xi}}\\
&\leq\cdots\\
&\leq \#\mathcal Q_n(k_1)\cdot2^{\frac{1}{\xi}(k_2+\cdots+k_s)}\\
&\leq  2^{\frac{1}{\xi}(k_1+\cdots+k_s)+2}.\end{align*}
The above counting argument shows that
the number of all feasible combinations $(k_1,\ldots,k_s)$ is $\le\sigma_1^{\frac{\varepsilon}{3}n}$.
Hence we obtain
$S(n)\leq (n/N)\sigma_1^{\frac{\varepsilon}{3}n}2^{\frac{n}{\xi}},$ and thus
$$\displaystyle{\varlimsup_{n\to\infty}}\frac{1}{n}\log S(n)\leq \frac{\varepsilon}{3}\log\sigma_1+\frac{1}{\xi}\log2
\leq\varepsilon.$$
The last inequality follows from the definition of $\xi$ in \eqref{xii}.\qed
\medskip

\subsection{Dimension of the set of points in open sets in $W^u$ not returning to $\Lambda$}\label{ncontrol}
Using Lemma \ref{bad} we estimate the Hausdorff dimension of the set of points in open sets in $W^u$ which 
do not return to $\Lambda$ under any positive iteration.
\begin{lemma}\label{esr0}
For any open set $\gamma$ in $W^u$, $\dim^u_H(\gamma\cap
B)\leq\varepsilon.$
\end{lemma}
\begin{proof}
To estimate the Hausdorff dimension of the set $\gamma\cap B$ we decompose it into the union of points belonging to $\Lambda$, points of the complement whose orbit eventually belongs to $\Lambda$ and points of the complement which never enter $\Lambda$. The proof of Lemma \ref{bad} allows us to obtain an upper estimate of the Hausdorff dimension of this subset.
For the complement of this subset we need a different computation.

\begin{sublemma}\label{dimB}
For each $\gamma^u\in\tilde\Gamma^u$ we have
$\dim^u_H(\gamma^u\cap B)\leq\varepsilon$.
\end{sublemma}
\begin{proof}
Let $z\in\gamma^u\cap\Omega$.

We call $l>0$ a {\it close return time of $z$} if $l=\min\{i>0\colon f^iz\in\Theta_{i}\}$. Points in $\Lambda$ have no close returns. 
Let $l_1,l_2,\ldots$ be
defined inductively as follows: $l_1$ is the close return time
of $z$; given $l_1,\ldots,l_{k-1}$, let $l_{k}$ be the close
return time of $f^{l_1+\cdots+l_{k-1}}z$.
Obviously $l_k\geq\xi l_{k-1}+N$ and $l_1\geq1$, and so
\begin{equation}\label{close}l_k\geq\xi^{k-1}.\end{equation}
 Let $\Xi_k$ denote the set of $z\in\gamma^u$
 for which $l_1,\ldots,l_k$
are defined in this way. Set $\Xi_\infty :=\bigcap_{k\geq1}\Xi_k$.
Since $\gamma^u\setminus\Xi_{1}=\gamma^u\cap\Lambda$, one has that
$$\gamma^u\cap B=\left(\gamma^u\cap\Lambda_B\right)\cup \left(\bigcup_{k\ge 1}\Xi_k\setminus\Xi_{k+1}\cap B\right)\cup \left(\Xi_\infty\cap B\right),$$
and so
$$
\dim^u_H(\gamma^u\cap B)\leq\max\left\{\dim^u_H\left(\Lambda_B\cap\gamma^u\right), \sup_{k\geq1}\dim^u_H\left((\Xi_k\setminus\Xi_{k+1})\cap B
\right),\dim^u_H(\Xi_\infty\cap B)\right\}.
$$
By Lemma~\ref{bad} $\dim^u_H\left(\gamma^u\cap\Lambda_B \right)=d^u(\Lambda_B)\leq\varepsilon$. In addition, for each $k\geq1$ the set
$\Xi_k\setminus\Xi_{k+1}$ can be decomposed into a countable collection of connected components
$\{z\in{\gamma^u}\colon f^{m+n}z\notin\Theta_{n}\ \forall n\geq0\}\cap B$, for some $m\geq0$. 
In the same way as in the proof of Lemma \ref{exceptional}, one can show that
$\dim^u_H\left((\Xi_k\setminus\Xi_{k+1})\cap B
\right)\leq\varepsilon$
and
$$\dim^u_H(\gamma^u\cap B)\leq\max\{\varepsilon,\dim^u_H (\Xi_\infty)\}.
$$

To show $\dim^u_H(\Xi_\infty)\leq\varepsilon$,
let $\mathcal U_k$ denote the collection of 
connected components of $\Xi_k$.
For each $u_k\in\mathcal U_k$ there exist a sequence
$l_1<\cdots<l_k$ of positive integers 
and a nested sequence $u_1\supset\cdots\supset u_k$ 
of curves such that for each $i=1,\ldots,k$,
 $f^{l_1+\cdots+l_i}u_i$ is a $C^2(b)$-curve
connecting the two stable sides of $\Theta_{l_i}$.
For $u_{k-1}\in\mathcal U_{k-1}$ and $l_k>0$ let
$$\mathcal R(u_{k-1},l_k)=\{u_k\in \mathcal U_{k}\colon
\text{$l_k$ is a close return time of
points in $f^{l_1+\cdots+l_{k-1}}u_{k}$}\}.$$ By definition, $$\Xi_k=
\bigcup_{u_{k-1}\in\mathcal U_{k-1}}\bigcup_{l_k}\bigcup_{u_k\in\mathcal R(u_{k-1},l_k)
}u_{k},$$ where the second
union runs over all possible $l_k$. 
For each $u_k\in\mathcal R(u_{k-1},l_k)$, let
$\widehat u_k$ denote the curve in $u_{k-1}$
containing $u_k$ such that
$f^{l_1+\cdots+l_k}\widehat u_k\in\tilde\Gamma^u$. 
Since the distortion of $f^{l_1+\cdots+l_k}|\widehat u_k$
is uniformly bounded by Lemma \ref{markov}, we have
$$\frac{\ell(u_k)} {\ell(u_{{k-1}})}\leq\frac{\ell(u_k)} {\ell(\widehat u_k)}\leq
C\frac{\ell(f^{l_1+\cdots+l_k}u_k)}
{\ell(f^{l_1+\cdots+l_k}\widehat u_k)}\leq
C\sigma_1^{-\xi l_k}.$$
Set
$\eta:=2\sigma_1^{-\varepsilon\xi}.$
By the definition of $\xi$ in \eqref{xii}, $\eta<1$.
Using $\#\mathcal R(u_{k-1},l_k)\leq 2^{l_k}$ and \eqref{close}
we get
$$\sum_{l_k}\sum_{u_k\in\mathcal R(u_{k-1},l_k)}
\frac{\ell(u_k)^\varepsilon}{\ell(u_{k-1})^\varepsilon} \leq C^\varepsilon\sum_{l_k\geq \xi^{k-1}}
2^{l_k}\sigma_1^{-\varepsilon\xi l_k}\leq
C^\varepsilon\eta^{\xi^{k-1}}.$$ Hence
\begin{align*}\sum_{u_k\in\mathcal U_k}\ell(u_{k})^\varepsilon&=\sum_{u_{k-1}
\in\mathcal U_{k-1}}
\ell(u_{k-1})^\varepsilon\left(\sum_{l_k}\sum_{u_k\in\mathcal R(u_{k-1},l_k)   }
\frac{\ell(u_{k})^\varepsilon} {\ell(u_{k-1})^\varepsilon}\right)\leq
C^{\varepsilon}\eta^{\xi^{k-1}}\sum_{u_{k-1}\in\mathcal U_{k-1}}\ell(u_{k-1})^\varepsilon.\end{align*} 
Using this recursively yields
$$\sum_{u_k\in\mathcal U_k}\ell(u_{k})^\varepsilon\leq
C^{\varepsilon(k-1)}\eta^{\sum_{i=1}^{k-1}\xi^{i-1}}
\sum_{u_{1}\in\mathcal U_{1}}\ell(u_{1})^\varepsilon.$$
The right-hand-side goes to $0$ as $k\to\infty$, and thus
the Hausdorff $\varepsilon$-measure of $\Xi_\infty$ is $0$.
\end{proof}
A successive use of Lemma \ref{markov} implies that any point in
 $\gamma\cap \Omega$ is contained in a curve in $W^u$ which
 is mapped by some forward iterates
to a curve in $\tilde\Gamma^u$. Hence, the countable stability of Hausdorff dimension and 
Sublemma \ref{dimB} complete the proof of Lemma \ref{esr0}.
\end{proof}



\subsection{Proof of (P6)}\label{pos}
We recall a few general results
on stable and unstable manifolds of
nonuniformly hyperbolic systems from \cite{Pes76,Rue79}.
Since
 any $\mu\in\mathcal M^e(f)$ has one positive and one negative 
Lyapunov exponent \cite{CLR08}, these general results hold
for our system.

For any $\mu\in\mathcal M^e(f)$ there exist Borel subsets $\Gamma_1\subset\Gamma_2\subset\cdots\subset \Omega$ 
such that ${\rm supp}(\mu)=
\Gamma_\infty:=\bigcup\Gamma_n$
 and sequences of positive
numbers $\delta_n\gg \epsilon_n$, possibly $\to0$ as $n\to\infty$,
such that, for $x\in\Gamma_n$:
\begin{itemize}

\item[(N1)] the {\it unstable manifold} $W^u(x)$ 
of $x$ (see \eqref{unsm})
is an injectively immersed $C^2$ submanifold with
$T_xW^u(x)=E^u(x)$. An analogous statement holds for the {\it
stable manifold} $W^s(x)$.
\end{itemize}

Let $B^u_\delta(x)$ (resp. $B^s_\delta(x)$) denote
the ball of radius $\delta$ centered at the origin
 of $ T_x\mathbb R^2$ in $E^u(x)$ (resp. $E^s(x)$)
and $B_\delta(x):=B^u_\delta(x)\times B^s_\delta(x)$.
Let $\Gamma_n(x):=\{y\in\Gamma_n\colon
|x-y|<\epsilon_n\}$ and for $y\in\Gamma_n(x)$, let $W^u_x(y)$ denote
the connected component of
$\exp_x^{-1}\left(W^u(y)\cap\exp_x(B_{\delta_n}(x))\right)$ that
contains $\exp_x^{-1}y$. 
\medskip

\begin{itemize}

\item[(N2)] 
For all $y\in\Gamma_n(x)$,
$W^u_x(y)$ is the graph of a function $\varphi\colon
B_{\delta_n}^u(x)\to B_{\delta_n}^s(x)$ with
$\|D\varphi\|\leq\frac{1}{100}$, for a conveniently chosen
metric.
An analogous
statement holds for $W^s_x(y)$.

\item[(N3)]
For $z\in\bigcup_{y\in\Gamma_n(x)}W^s_x(y)$,
let $\mathcal F^s(z)$ denote the element of $\{W^s_x(y)\}_{y\in \Gamma_n(x)}$ which contains $z$.
Then $z\mapsto T_z\mathcal F^s(z)$ is
Lipschitz continuous.

\item[(N4)] 
The holonomy map $\pi\colon
\Sigma_1\cap\bigcup_{y\in\Gamma_n(x)}W_x^s(y)\to\Sigma_2 $
defined by $\pi(y)=W_x^s(y)\cap\Sigma_2$ for any 
graph 
$\Sigma_i$  $(i=1,2)$ of 
a $C^1$ function $\psi_i\colon B_{\delta_n}^u(x)\to
B_{\delta_n}^s(x)$ with $\|D\psi_i\|\leq\frac{1}{99}$
is bi-Lipschitz continuous. 
In particular, it preserves Hausdorff dimension.
\end{itemize}

\begin{remark}\label{holono}
{\rm 
Since ${\rm dim}\, E^u=1$  the
constant $\alpha$ in the {\it bunching condition} \cite[(19.1.1)]{KatHas95}
can be taken to be $1$. Then
(N3) follows from a slight modification of the proof of
\cite[Theorem 19.1.6]{KatHas95}. (N4) follows from (N3) and the fact that 
${\rm dim}\, E^s=1$.}
\end{remark}

Let $x\in\Gamma_\infty$.
 For each $n>0$
 consider a countable covering
 $\{\Gamma_n(z_i)\}_i$ of $\Gamma_n\cap W^u(x)$
 such that 
$\bigcup_{i} W^u_{\rm loc}(z_i)=\Gamma_n\cap W^u(x),$
where
$W_{\rm loc}^u(z_i):=\exp_{z_i} W^u_{z_i}(z_i)$.
Let
$B_i=W^u_{\rm loc}(z_i)\cap B$.
\begin{lemma}\label{dim1}
$\dim^u_H(B_i)\leq\varepsilon$.
\end{lemma}
\begin{proof}
By Katok's closing lemma \cite[Theorem S.4.13]{KatHas95}, there exists 
 a periodic saddle
$p_i\in\Gamma_n(z_i)$ such that
$W_{z_i}^u(p_i)$ is the 
graph of a function
$\varphi\colon B_{\delta_{n}}^u(z_i)\to B_{\delta_{n}}^s(z_i)$
with $\|D\varphi\|\leq\frac{1}{100}$. Since $W^s(p_i)$ and $W^u$ have
transverse intersections, the Inclination Lemma implies the existence
of a connected component of $\exp_{z_i}^{-1}(W^u)\cap B_{\delta_{n}}(z_i)$ that is
the graph of a function $\psi\colon B_{\delta_{n}}^u(z_i)\to
B_{\delta_{n}}^s(z_i)$ with $\|D\psi\|\leq\frac{1}{99}$. 
Let $\pi$ denote the holonomy map between $W^u_{\rm loc}(z_i)$ and $\exp_{z_i}(\rm{graph}(\psi))$.

\begin{sublemma}\label{esr1} $\pi(x)\in B$ if and only if $x\in B$.
\end{sublemma}
\begin{proof}
If $x\notin B$ there
exist $k\geq0$ and $\gamma^s\in\Gamma^s$ such that
$f^kx\in\gamma^s$. We have
$f^kW^s_{\rm loc}(x)\subset W^s_{\rm loc}(f^kx)$ and
$\gamma^s\subset W^s(f^kx)$.
We have
$W^s_{\rm loc}(f^kx)\subset\gamma^s$,
for otherwise $W^s_{\rm loc}(f^kx)$
contains points that escape to infinity.
Since both $x$ and $\pi(x)$ belong to $W^s_{\rm loc}(x)$ then $f^k(\pi (x))\in\Gamma^s$ so $\pi(x)\notin B$. The same reasoning yields the converse.
\end{proof}

By Sublemma \ref{esr1}, $\pi(B_i)\subset B$
and Lemma \ref{esr0} gives $\dim^u_H(\pi(B_i))\leq\varepsilon$.
(N4) yields $\dim^u_H(B_i)\leq\varepsilon$. 
\end{proof}

\noindent{\it Proof of (P6).}
Take $\mu\in\mathcal M^e(f)$ with $h(\mu)\geq 2\varepsilon$.
Observe that since $\Gamma_n\cap W^u(x)\cap B\subset \bigcup_{i} B_i$, Lemma \ref{dim1} yields
 $\dim^u_H(\Gamma_n\cap W^u(x)\cap B)\leq\varepsilon$ for every $n>0$,
 and thus 
 $\dim^u_H(\Gamma_\infty\cap W^u(x)\cap B)\leq\varepsilon$.
Let $\{\mu^u_x\}_{x\in\Gamma_\infty}$ denote the
canonical system of conditional measures of $\mu$ along
unstable manifolds.
The dimension formula \cite{LedYou85} gives
$\dim(\mu^u_x)=\dim^u(\mu)=\frac{h(\mu)}{\lambda^u(\mu)}>\varepsilon$,
and thus 
$\mu^u_x(W^u(x)\cap\Gamma_\infty\cap  B)<1$ and 
$\mu^u_x((W^u(x))^c\cup\Gamma^c_\infty\cup B^c)>0$. Since $\mu^u_x((W^u(x))^c)=0=\mu^u_x(\Gamma_\infty^c)$ we have $\mu(B^c)=\int_{x\in\Gamma_\infty}\mu^u_x( B^c)d\mu(x)>0$.
The $f$-invariance of $\mu$ yields $\mu(\Lambda)>0.$

Poincar\'e recurrence gives $\mu(X)>0$. Since $F$ is the first
return map to $X$, Kac's formula \cite[Theorem 2']{Kac47} gives
$\int_{X} \tau d\mu=1,$ and so $\tau$ is
$\mu$-integrable. By \cite{Zwe04},
$\mu$ is liftable.
 $\hfill\qed$

\section{Proofs of the theorems}
In this last section we prove Theorems A, B, C, D.
In Sect.\ref{ss} we show that the induced potential
$\overline{\varphi_t}\colon X\to\mathbb R$ has
strongly summable variations and finite Gurevich pressure.
Prior to Theorem A, we  prove Theorem B in Sect.\ref{uhd}.
In Sect.\ref{posrec} we define two numbers $t_-<0<t_+$ and show that $\overline{\varphi_t}$ is positive recurrent
for any $t\in(t_-,t_+)$.
From Proposition \ref{mainthemorem}
it follows that for any $t\in(t_-,t_+)$ there exists a unique measure which maximizes $F_{\varphi_t}$ among measures which are liftable to the inducing scheme. In Sect.\ref{range} we complete the proof of Theorem A by showing that this candidate measure is indeed an equilibrium measure for $\varphi_t$.
Theorem C and Theorem D are proved in Sect.\ref{c} and Sect.\ref{d}.

\subsection{Strong summability and finite Gurevich pressure}\label{ss}

By (P2), the unstable direction $E^u$ (see the definition in \eqref{eu}) makes sense on each $\gamma^u\in\Gamma^u$,
and coincides with its tangent directions.
Hence, for each $z\in X$  we have $\sum_{i=0}^{\tau(J)-1}J^u(f^iz)=\|D_zF|E_z^u\|$. We now prove that the induced potential function 
$\overline{\varphi_t}(z):=-t\log\|D_zF|E_z^u\|$
has strongly summable variations (i.e. the potential $t\Phi=\overline{\varphi_t}\circ h$ has strongly summable variations).

\begin{lemma}\label{summable}
There exists $C>0$ such that for every $n>0$, $V_n(\Phi)\leq Cb^{-1}\sigma_1^{-n}.$ In particular, $t\Phi$ has strongly summable variations
for any $t\in\mathbb R$.
\end{lemma}
\begin{proof}
Take $\underline{a},\underline{a}' \in S^{\mathbb Z}$ such that
$a_i=a_i'$ provided $|i|\leq n-1$. Let $x=h(\underline{a})$, $x'=h(\underline{a'})$.
Let $y$ denote the point of intersection between $\gamma^u(x)$ and $\gamma^s(x')$. 
We have \begin{align*}
|\Phi(\underline{a})-\Phi(\underline{a'})|=\left|
\log \frac{\|D_xF|E^u_x\|}{\|D_{x'}F|E^u_{x'}\|}\right|\le 
\left|\log \frac{\|D_xF|E^u_x\|}{\|D_yF|E^u_y\|}\right|+ 
 \left|\log\frac{\|D_yF|E^u_y\|}{\|D_{x'}F|E^u_{x'}\|}\right|.
\end{align*}
By the $F$-invariance of the $\gamma^s$-curves, 
$F^iy\in\gamma^s(F^ix')$ for $ 0\leq i\leq n-1$. 
By the $F$-invariance of the $\gamma^u$-curves, 
$F^ix$ and $F^iy$ belong to the same $\gamma^u$-curves for $0\leq i\leq n-1$.
(P3) implies
$|Fx-Fy|\leq
\sigma_1^{-n}$.
Using this and (P4)(a) we have
\begin{equation}\label{sumeq1}
\left|\log \frac{\|D_xF|E^u_x\|}{\|D_yF|E^u_y\|}\right|
\leq
 C|Fx-Fy|\leq
C\sigma_1^{-n}.\end{equation}

For $z\in \Gamma^u$ let $e^u(z)$ denote the unit vector with a positive first component which spans 
$E^u_z$. 
From the bounded distortion in (P4)(b) and the proof of Lemma \ref{leaf} in Appendix A2 we have
$\|D_zf^je^u(z)\|\geq(1/2)\kappa^{j}$
for every $j\geq1$, where  $\kappa=5^{-(1+\xi)N}$.
Then the angle estimate in
\cite[Claim 5.3]{WanYou01} yields
$$\angle(D_yf^ie^u(y),D_{x'}f^ie^u(x'))\leq (Cb)^{\frac{i}{2}}\angle(e^u(y),e^u(x'))\leq (Cb)^{\frac{i+n}{2}}.$$
From the contraction along the $\gamma^s$-curves we have
$$\| D_{f^iy}f-D_{f^ix'}f  \|\leq C|f^iy-f^ix'|\leq(Cb)^{\frac{i}{2}}|y-x'|\leq (Cb)^{\frac{i+n}{2}}.$$
Hence
\begin{align*}\left| \log\frac{J^u
(f^{i}y)}{J^u(f^ix')}\right|&\leq Cb^{-1}
\left\|\frac{D_yf^{i+1}e^u(y)}{\|D_yf^ie^u(y)\|}-
\frac{D_{x'}f^{i+1}e^u(x')}{\|D_{x'}f^ie^u(x')\|}\right\|\\
&\leq Cb^{-1}\left(\| D_{f^iy}f-D_{f^ix'}f  \|+C\angle(D_yf^ie^u(y),D_{x'}f^ie^u(x')) \right)\\
&\leq (Cb)^{\frac{i+n}{2}-1}.
\end{align*}
The first inequality follows from the fact that $|\log(1+\psi)|\leq|\psi|$ for $\psi\geq0$ and $J^u\geq b/5$.
The second one follows from the triangle inequality. Then
\begin{equation}\label{sumeq2}\left|\log\frac{\|D_yF|E^u_y\|}{\|D_{x'}F|E^u_{x'}\|}\right|=\sum_{i=0}^{\tau(x)-1}\left| \log\frac{J^u
(f^{i}y)}{J^u(f^ix')}\right|\leq\sum_{i=0}^{\tau(x_0)-1}
(Cb)^{\frac{i+n}{2}-1}\leq (Cb)^{\frac{n}{2}-1}.\end{equation}
\eqref{sumeq1} \eqref{sumeq2} yield the desired inequality.
\end{proof}

We show the finiteness of the Gurevich pressure of the induced potential of a \emph{"shifted"} potential.
For $t,c\in\mathbb R$ define
$$T_{t,c}=\sum_{J\in S}e^{c\tau(J)}
\ell(J)^{t}$$
and
\begin{align*}
c_0(t)=\begin{cases}
& t\log \sigma_2-
\displaystyle{\varlimsup_{n\to\infty}}(1/n)\log
S(n)\ \ \text{if} \ \ t<0;\\
&t\log\sigma_1-
\displaystyle{\varlimsup_{n\to\infty}}(1/n)\log
S(n)\ \ \text{if} \ \ t\geq0.
\end{cases}\end{align*}
By (P3)(a), for some $C>0$ we have
\begin{equation}\label{ttc}
T_{t,c}\leq
\begin{cases}
C\sum_{n>1}S(n)e^{cn}\sigma_2^{-tn}
\text{ if }t<0;&\\
C\sum_{n>1}S(n)e^{cn}\sigma_1^{-tn}
\text{ if }t\geq0.
\end{cases}
\end{equation}

\begin{lemma}\label{finiteG}
If
$c<c_0(t),$ then $T_{t,c}<\infty$ and thus
$P_G\left(\overline{\varphi_t+c}\right)<\infty$.
\end{lemma}
\begin{proof}
In the case $t\geq0$, using the second alternative of \eqref{ttc} and (P5) we have
$$ T_{t,c}\leq
C\sum_{n>1}
\exp\left(n \left(c-t\log\sigma_1+\frac{1}{n}\log S(n)\right)\right)
<\infty.$$
The case $t<0$ can be handled similarly.

As for the Gurevich pressure, fix $J_0\in S$. Observe that $\overline{\varphi_t+c}=
-t\log\|Df^{\tau}|E^u\|+c\tau$ and so, 
\begin{align*}
P_G\left(\overline{\varphi_t+c}\right)&=\lim_{n\to\infty}\frac{1}{n}\log\sum_{
\stackrel{x\in \gamma^u(\zeta_0)\cap J_0}{F^nx\in\gamma^s(x)}}\exp\left(\sum_{i=0}^{n-1}
\overline{(\varphi_t+c)}(F^ix)\right)\\
&\leq \lim_{n\to\infty}\frac{1}{n}\log\left(\sum_{J\in S}\sup_{x\in J}\exp\overline{(\varphi_t+c)} (x)\right)^n\leq\lim_{n\to\infty}\frac{1}{n}\log (C\cdot T_{t,c})^n=\log T_{t,c}<\infty, \end{align*}
where $C>0$ is a uniform constant.
\end{proof}

\subsection{Unstable Hausdorff dimension of $\Omega$}\label{uhd}
Before proceeding to the proof of Theorem B we need a couple of lemmas.
\begin{lemma}\label{tulow}
$t^u>\frac{\log2}{\log5}$.
\end{lemma}
\begin{proof}
Consider the line through the two points $(0,\log2)$ and
$(t^u,0)$ on the pressure curve $\{(t,P(t))\in\mathbb R^2\colon t\in\mathbb R\}$. 
The point $(-1,(1/t^u)\log2+\log2)$ lies on this line.
Since the pressure curve is concave up,  we
have $(1/t^u)\log2+\log2\leq P(-1)$.
Since $\|Df\|<5$ on $R$ we have $P(-1)<\log 2+\log5$,
and thus the desired inequality holds.
\end{proof}

For $\mu\in\mathcal M(f)$, let $$\lambda^u(\mu)=\int\log J^ud\mu.$$ 
A proof of the next lemma is given in Appendix A3.
\begin{lemma}\label{Lyap}
$\inf\{\lambda^u(\mu)\colon\mu\in\mathcal M^e(f)\}\geq\log(2-\varepsilon)$.
\end{lemma}

\noindent{\it Proof of Theorem B.} 
Take an open set $\gamma$ in $W^u$ intersecting $\Omega$. 
For each $z\in\gamma\cap\Omega$ there exists a curve $\gamma(z)\subset\gamma$ and an integer $n(z)>0$
such that $f^{n(z)}\gamma(z)\in\tilde\Gamma^u$.
 The set $(f^{n(z)}\gamma(z)\cap \Omega)\setminus
B$ is decomposed into a countable collection of sets which are sent by some
positive iterates to 
sets of the form $\gamma^u\cap \Lambda$, $\gamma^u\in\tilde\Gamma^u$. 
By the countable stability and the invariance of
the Hausdorff dimension under the action of bi-Lipschitz homeomorphisms,
we have $\dim^u_H(
f^{n(z)}\gamma(z)\cap \Omega)\setminus B)=d^u(\Lambda)$.
We also have $d^u(\Lambda)=\max\{d^u(X),d^u(\Lambda_B)\}$, 
and thus
\begin{align*}
\dim^u_H(
\gamma(z)\cap \Omega)&=\dim^u_H(
f^{n(z)}\gamma(z)\cap \Omega)\\
&=\max\{\dim^u_H(
f^{n(z)}\gamma(z)\cap \Omega)\setminus B),\dim^u_H(
f^{n(z)}\gamma(z)\cap \Omega)\cap B)\}\\
&=\max\{d^u(\Lambda),\dim^u_H(
f^{n(z)}\gamma(z)\cap \Omega)\cap B)\}\\
&= \max\{d^u(X),d^u(\Lambda_B),\dim^u_H(
f^{n(z)}\gamma(z)\cap \Omega)\cap B)\}.\end{align*}
In the next two paragraphs we show $d^u(X)=t^u$. This and
Lemma \ref{tulow} imply  $\dim^u_H(
\gamma(z)\cap \Omega)=t^u.$
Since $\gamma$ is decomposed into a countable number of curves like $\gamma(z)$,
the countable stability of Hausdorff dimension yields $\dim^u_H(\gamma\cap \Omega)=t^u.$ 
Hence the first statement of Theorem B holds.
\medskip

We are in position to show $d^u(X)=t^u$. Fix a basic element $J_0\in S$. Consider the covering $\mathcal U_n$ of $\gamma^u(\zeta_0)\cap J_0$ by $n$-cylinders. 
Using (P3)(a) and (P4)(a), for some $C>0$ and all $t>0$ we have
\begin{align*}
\sum_{U\in\mathcal U_n}\ell(U)^t\leq C^t
 \sum_{\stackrel{x\in\gamma^u(\zeta_0)\cap J_0}{ F^nx\in\gamma^s(x)} }\exp
 \left(-t\sum_{i=0}^{n-1}\log \|DF|E^u(F^ix)\|\right).
\end{align*}
By definition the expression of the right-hand-side has the growth rate $P_G(\overline{\varphi_t})$ as $n$ increases.
Since the pressure is non-increasing and $t^u$ is the unique solution of the equation $P(t)=0$, $P(t)<0$ holds for all $t>t^u$.  For these $t$, $c_0(t)>0$, and thus by Lemma \ref{finiteG}, $\overline{\varphi_{t}}$ has finite Gurevich pressure.
It has strongly summable variations by Proposition~\ref{summable}, and hence, there exists a unique $F$-invariant Gibbs measure $\nu_{\overline{\varphi_t}}$ for $\overline{\varphi_{t}}$.
We also have
$\nu_{\overline{\varphi_t}}(\tau)<\infty$. The Variational Principle and Abramov's and Kac's formul{\ae} \cite[Theorem 2.3]{PesSen08} yield
$P_G(\overline{\varphi_t})<0$.
Hence the
Hausdorff $t$-measure of $\gamma^u(\zeta_0)\cap J_0$ is $0$. Since
$t>t^u$ is arbitrary, 
$d^u(X)=\dim^u_H(\gamma^u(\zeta_0)\cap
J_0)=\dim^u_H(\gamma^u(\zeta_0)\cap
X)\leq t^u$.

To show the reverse inequality, 
pick an ergodic equilibrium measure for $\varphi_{t^u}$, 
which was proved to exist in \cite[Theorem]{SenTak11} and denote it by $\mu_{t^u}$. 
The dimension formula gives $h(\mu_{t^u})=\dim^u(\mu_{t^u})\lambda^u(\mu_{t^u})$.
Using the equation 
$F_{\varphi_{t^u}}(\mu_{t^u})=0$, $\varepsilon\ll1$ and Lemma \ref{tulow} 
we have $\dim^u(\mu_{t^u})=t^u>4\varepsilon$.
From this and Lemma \ref{Lyap} we have
$h(\mu_{t^u})\geq2\varepsilon$. By (P6),
$\mu_{t^u}$ is liftable.
Let $\{\nu_x\}_{x}$ denote the canonical system
of conditional measures of $\mu_{t^u}$ along unstable manifolds.
Since $\mu_{t^u}$ gives full weight to the set $Y:=\bigcup_{n\geq0} f^nX$,
$\nu_x(W^u(x)\cap Y)=1$ holds for $\mu_{t^u}$-a.e. $x$.
(P2) gives
$\gamma^u(x)\subset W^u(x)$, and thus $W^u(x)\cap Y=\bigcup_{n\geq0}f^n(\gamma^u(x)\cap X)$.
Since $\dim^u(\mu_{t^u})={\dim}(\nu_x)=t^u$ we have
$\dim^u_H(\gamma^u(x)\cap X)\geq t^u$, and therefore $d^u(X)\geq t^u$.
\medskip

To complete the proof of Theorem B it is left to show $t^u\to1$ as $b\to0$.
Define a decreasing sequence $\{E_k\}$ of compact sets inductively by
$E_0=\gamma^u(\zeta_0)$ and $E_k=\overline{E_{k-1}\setminus f^{-k+1}
\Theta_0}$ for $k\geq1$. 
Set $E_\infty:=\bigcap _{k=0}^\infty E_k$. This set is similar in spirit to $\Omega_\infty$, but its Hausdorff dimension is easier to estimate because one removes a fixed core at each step. Observe that 
$E_\infty=\overline{\gamma^u(\zeta_0)\setminus\bigcup_{i=0}^\infty f^{-i}
\Theta_0}
\subset\gamma^u(\zeta_0)\cap \Omega$, and $E_\infty$ is a Cantor set in $\gamma^u(\zeta_0)$.

Let $\mathcal E_k$ denote the collection of components of $E_k$.
For each $A\in\mathcal E_k$ choose a point $x_A\in A\cap E_\infty$
and denote by $\mu_k$ the atomic probability measure which 
is uniformly distributed on the set $\{x_A\colon A\in\mathcal E_k\}$.
Pick a limit point of $\{\mu_k\}$ and denote it by $\mu_\infty$.
Since  $E_\infty$ is closed, $\mu_\infty(E_\infty)=1$.
 By construction, for every $A\in \mathcal E_k$ and $p\geq k$ we have
$$\mu_p(A)=\frac{\#\{B\in\mathcal E_p\colon B\subset A\}}{\#\mathcal E_p}
=\frac{1}{\#\mathcal E_k}.$$
Since $\mu_\infty$ assigns no weight to the endpoints of $A$,
 $\mu_p(A)\to\mu_\infty(A)$ as $p\to\infty$. Hence
 \begin{equation}\label{tuweq1}
 \mu_\infty(A)=\lim_{p\to\infty}\mu_p(A)=\frac{1}{\#\mathcal E_k}.\end{equation}

\begin{lemma}\label{lengthf}
There exist constants $C_N>0$, $C_b$ such that 
for every $k\ge1$ and $A\in\mathcal E_k$, 
$(1/C_b)(2+\varepsilon)^{-k}\leq \ell(A)\leq (1/C_N) (2-\varepsilon)^{-k}.$
\end{lemma}

\begin{proof}
Since $b\ll1$, $f$ may be viewed as a small perturbation of the Chebyshev quadratic polynomial $x\in[-1,1]\to1-2x^2$, which
is topologically conjugate to the tent map with slope $\pm2$. Since the conjugacy is smooth
except at the boundary points where the derivative blows up,
the following holds for $f$:
\begin{itemize}
\item[(a)]
there exists a constant $C_N>0$ such that
if $i\geq1$, $z\in \gamma^u(\zeta_0)\setminus \Theta_0$
and $fz,\ldots,f^{i-1}z
\notin \Theta_0$,
then $C_N(2-\varepsilon)^i\leq\|Df^i|E^u(z)\|\leq C_b(2+\varepsilon)^i$.
\end{itemize}
In addition, from (a) and \cite[Lemma 2.4]{WanYou01}, 
\begin{itemize}
\item[(b)]  if $\gamma\subset R\setminus \Theta_0$
is a $C^2(b)$-curve, then $f\gamma$ is $C^2(b)$.\end{itemize}
Let $A\in\mathcal E_{k}$. By construction, $f^iA$ does not intersect $\Theta_0$ for every $0\leq i\leq k-1$. In particular, $f^{k-1}A$ is a $C^2(b)$-curve by (b). If $f^{k-1}A$ is at the left of $\Theta_0$, then the left endpoint of $f^{k-1}A$ is in $\tilde\alpha_j$ for some $1\leq j\leq N-1$ and the right endpoint of it is in $\alpha_N^-$. If $f^{k-1}A$  is at the right of $\Theta_0$, the left endpoint of $f^{k-1}A$ is in $\alpha_{N}^+$ and $f^{k-1}A$ intersects $\alpha_1^+$. In both cases $f^kA$ is a $C^2(b)$-curve intersecting both $f(\alpha_{N}^-\cup\alpha_{N}^+)$ and $\alpha_1^+$. Therefore $\ell(f^{k}A)>L$ for some constant $L>0$ and (a) yields the desired inequalities.
\end{proof}

Let ${\rm diam}(\Theta_0)$ denote the diameter of $\Theta_0$.
From Lemma \ref{markov}(b)
there exists $C>0$ independent of $\varepsilon,N,b$ such that
for every $k\geq0$,
$$\frac{\ell(E_k\setminus E_{k+1})}{\ell(E_k)}\leq C {\rm diam}(\Theta_0).$$
This yields
$\ell(E_{k+1})=\ell(E_{k})-\ell(E_k\setminus E_{k+1})\geq(1-C {\rm diam}(\Theta_0))\ell(E_{k})
$.
Using this inductively,
\begin{equation}\label{tueq2}
\ell(E_k)\geq (1-C {\rm diam}(\Theta_0))^k.
\end{equation}
Set $\rho=(2-\varepsilon)(1-C {\rm diam}(\Theta_0))$.
By \eqref{tueq2} and Lemma \ref{lengthf},
\begin{equation}\label{tueq3}
\#\mathcal E_k\geq\frac{\ell(E_k)}{\sup\{\ell(A)\colon A\in\mathcal E_k\}}
\geq C_N\rho^k.\end{equation}

To finish, let $U$ be a small curve in $\gamma^u(\zeta_0)$. Choose a large integer $k>0$ such that
\begin{equation}\label{tueq4}
(1/C_b)\left(2+\varepsilon\right)^{-k-1}<\ell(U)\leq(1/C_b) 
\left(2+\varepsilon\right)^{-k}.\end{equation}
By Lemma \ref{lengthf}, $U$ can intersect at most two elements of $\mathcal E_{k}$.
Using \eqref{tuweq1}, \eqref{tueq3}, \eqref{tueq4} and $\ell(U)<1$,
\begin{align*}
\mu_\infty(U)&\leq (2/\#\mathcal E_{k})\leq (2/C_N)\rho^{-k}=
(2/C_N)\left(2+\varepsilon\right)^{\frac{-k\log\rho}{\log(2+\varepsilon)}}\\
&
\leq(2/C_N) C_b^{ \frac{\log\rho}{\log(2+\varepsilon)}   }
\ell(U)^{ \frac{\log\rho}{\log(2+\varepsilon)}   }.\end{align*}
The Mass Distribution Principle \cite[p.60]{Fal90} yields
$\dim^u_H(E_\infty)\geq 
\frac{\log\rho}{\log (2+\varepsilon)}$. Note that this number can be taken arbitrarily close to $1$ at the cost of 
reducing $\varepsilon,1/N$ and $b$.
Since $E_\infty\subset\gamma^u(\zeta_0)\cap \Omega$
and $t^u=\dim^u_H(\gamma^u(\zeta_0)\cap \Omega)$ from the first statement of Theorem B,
we obtain $t^u\to1$ as $b\to0$.
\qed

\begin{cor}\label{lambdatu}
$\lambda^u(\mu_{t^u})\to\log2$ as $b\to0$.
\end{cor}
\begin{proof}
The topological entropy of $f$ is $\log2$. 
The relation $F_{\varphi_{t^u}}(\mu_{t^u})=0$ and the variational principle give
$\lambda^u(\mu_{t^u})\leq\log2/t^u.$
On the other hand, Lemma \ref{Lyap} gives $\lambda^u(\mu_{t^u})\geq
\log(2-\varepsilon)$.
Since $t^u\to1$ as $b\to0$ as in Theorem B and $\varepsilon>0$ can be made
arbitrarily small by choosing small $b$, we get the claim.
\end{proof}

\subsection{Positive recurrence}\label{posrec}
We now define $-1<t_-<0<t_+$ by
\begin{equation}\label{t+-}t_+=\frac{t^u\lambda^u(\mu_{t^u})}{\lambda^u(\mu_{t^u})-\log(2-\varepsilon)+\sqrt{\varepsilon}}
\quad\text{and}\quad t_-=\frac{t^u\lambda^u(\mu_{t^u})}{\lambda^u(\mu_{t^u})-\log(4+\varepsilon)-\sqrt{\varepsilon}}.\end{equation}
Corollary \ref{lambdatu}  implies that these definitions make sense. It also implies that
 one can make $t_+$ and $t_-$ arbitrarily large and close to $-1$ respectively,
by choosing sufficiently small $\varepsilon$.

\begin{lemma}\label{pr}
If $t\in(t_-,t_+)$, then $\overline{\varphi_t}$ is positive recurrent.
\end{lemma}
\begin{proof}
Let $\mathcal M_L(f)$ denote the set of liftable measures to the inducing scheme $(S,\tau)$
in Proposition \ref{lat}. Let 
 $$P_L(\varphi_t):=\sup\{F_{\varphi_t}(\mu)\colon\mu\in\mathcal
M_L(f)\}.$$
In view of Lemma \ref{finiteG} it suffices to show that one can choose $\eta_0>0$ so that 
$T_{t,-(P_L(\varphi_t)-\eta)}$ is finite for all $0\leq\eta\leq\eta_0$.
To show this we  first estimate $P_L(\varphi_t)$ from below.
In the proof of Theorem B we have shown that 
$\mu_{t^u}\in\mathcal M_L(f)$. Hence
\begin{equation}\label{pl}
P_L(\varphi_t)\geq
F_{\varphi_{t}}(\mu_{t^u})=h(\mu_{t^u})-t\lambda^u(\mu_{t^u})=(t^u-t)\lambda^u(\mu_{t^u}).\end{equation}

To show the finiteness of $T_{t,-(P_L(\varphi_t)-\eta)}$
we consider the following three cases.
\medskip

\noindent{\bf Case I:} $0<t^u\leq t< t_+.$
Using \eqref{pl} and the fact that $\sigma_1=2-\varepsilon$ in \eqref{sigma12} we have
$$-P_L(\varphi_t)-t\log\sigma_1+\frac{1}{n}\log
S(n)\leq(t-t^u)\lambda^u(\mu_{t^u})-t\log(2-\varepsilon)+\frac{1}{n}\log S(n).$$By the definition of $t_+$ in \eqref{t+-} and 
the bound on $S(n)$ from (P5),
the right-hand-side is strictly negative
for all large $n$. Therefore for sufficiently small $\eta\geq0$, 
\[
T_{t,-(P_L(\varphi_t)-\eta)} \leq C 
\sum_{n>0}\exp\left({n\left(-P_L(\varphi_t)+\eta-t\log\sigma_1+\frac{1}{n}\log
S(n)\right)}\right)<\infty.\]

\noindent{\bf Case II:} $0\leq t< t^u<1.$
Jensen's inequality applied to the convex function $x\to x^t$ yields
$$\sum_{\tau(J)=n}\ell(J)^t\leq S(n)^{1-t}\left(\sum_{\tau(J)=n}
\ell(J)\right)^t.$$
Using this we have
\begin{align*}
e^{-(P_L(\varphi_t)-\eta)n}\sum_{\tau(J)=n}\ell(J)^t
&\leq
\exp\left(
\left(
\eta+(t-t^u)\lambda^u(\mu_{t^u})
+(1-t)\frac1n\log S(n)-\frac{t}{2}\log\sigma_1
\right)
n\right)\\
&\leq\exp\left(\left(\eta-t^u\lambda^u(\mu_{t^u})+t\left(\lambda^u(\mu_{t^u})-\frac{1}{2}\log\sigma_1\right)+(1-t)\varepsilon\right)n\right).\end{align*}
Since $t^u\to1$ and $\lambda^u(\mu_{t^u})\to\log2$ as $b\to0$, the exponent
is strictly negative for sufficiently small $\eta\geq0$. Therefore
$T_{t,-(P_L(\varphi_t)-\eta)}<\infty$ holds.
\medskip

\noindent{\bf Case III:} $t_-< t\leq0.$
Using (\ref{PL}) and the fact that $\sigma_2=4+\varepsilon$ in \eqref{sigma12} we have 
$$-P_L(\varphi_t)-t\log \sigma_2+\frac{1}{n}\log
S(n)\leq(t-t^u)\lambda^u(\mu_{t^u})-t\log(4+\varepsilon)+\frac{1}{n}\log
S(n).$$
By the definition of $t_-$ in \eqref{t+-} and (P5), 
the right-hand-side is strictly negative for all large $n$.
Therefore
for sufficiently small $\eta\geq0$,
\[
T_{t,-(P_L(\varphi_t)-\eta)} \le C
\sum_{n>0}\exp\left(n\left(-(P_L(\varphi_t)-\eta)-t\log \sigma_2+\frac{1}{n}\log
S(n)\right)\right)<\infty.
\]
This completes the proof of Lemma \ref{pr}.
\end{proof}

\begin{cor}\label{uniqueL}
For any $t\in(t_-,t_+)$ there exists a unique equilibrium measure for $\varphi_t$ among all liftable measures.
\end{cor}
\begin{proof}
Choose $c<c_0(t)$ so that $-c\gg1$.
Then $\varphi_{t}+c$ has finite Gurevich pressure,
and is strongly summable by Proposition \ref{summable}. Observe
that $P_L(\varphi_{t}+c)=P_L(\varphi_{t})+c$ and so
$\overline{\varphi_{t}+c-P_L(\varphi_{t}+c)}=\overline{\varphi_{t}-P_L(\varphi_{t})}$.
Since $\varphi_{t}$ is positive recurrent by Lemma \ref{pr},
so is $\varphi_{t}+c$. By Proposition
\ref{mainthemorem}, there exists a Gibbs measure
$\nu_{\overline{\varphi_t+c}}$. By the Gibbs property, for any $J\in S$ and for all $x\in
J$,
\[\nu_{\overline{\varphi_t+c}}(J)\leq C\exp
\left(-P_G(\overline{\varphi_t+c})+\overline{\varphi_t+c}(x)\right)
\leq
Ce^{-P_G(\overline{\varphi_t+c})}e^{c\tau(J)}\max\left(\sigma_1^{-t\tau(J)},\sigma_2^{-t\tau(J)}\right),\]
and therefore
\begin{equation}\label{exptail}\sum_{\stackrel{J\in S}{\tau(J)=n}} \tau(J)\nu_{\overline{\varphi+c}}(J)\leq
C n S(n)e^{-P_G(\overline{\varphi_t+c})}e^{cn}\max\left(\sigma_1^{-tn},\sigma_2^{-tn}\right).
\end{equation}
The right-hand-side has a negative growth rate as $n$ increases.
Hence $\nu_{\overline{\varphi_t+c}}(\tau)<\infty$ holds. By
Proposition \ref{mainthemorem}, there exists a unique equilibrium
measure for $\varphi_{t}+c$ among all liftable measures.
Since $\varphi_t+c$ is cohomologous to $\varphi_t$, they yield the same equilibrium measures. 
\end{proof}

\subsection{Uniqueness of equilibrium measures for $\varphi_t$}\label{range}
 We finish the proof of Theorem A. 
 We start with preliminary estimates of $t_\pm$.
Define
$$\lambda_M^u:=\sup\{\lambda^u(\mu)\colon\mu\in\mathcal
M^e(f)\}\ \ \text{and}\ \
\lambda_m^u:=\inf\{\lambda^u(\mu)\colon\mu\in\mathcal
M^e(f)\}.$$
\begin{lemma}\label{t1'}
We have
\begin{equation*}
t_+<   \frac{t^u\lambda^u(\mu_{t^u})-2\varepsilon}{\lambda^u(\mu_{t^u})-\lambda^u_m}
\ \ \text{and}\ \ t_->\frac{
t^u\lambda^u(\mu_{t^u})-2\varepsilon}
{\lambda^u(\mu_{t^u})-\lambda^u_M}.\end{equation*}
\end{lemma}
\begin{proof}
A direct computation gives
$$\frac{t^u\lambda^u(\mu_{t^u})-2\varepsilon}{\lambda^u(\mu_{t^u})-
\lambda^u_m}-t_+=\frac{t^u\lambda^u(\mu_{t^u})(\lambda^u_m-\log(2-\varepsilon)+\sqrt{\varepsilon})-2\varepsilon(\lambda^u(\mu_{t^u})-\log(2-\varepsilon)+\sqrt{\varepsilon})}{(\lambda^u(\mu_{t^u})-\lambda^u_m)(\lambda^u(\mu_{t^u})-\log(2-\varepsilon)+\sqrt{\varepsilon})}.$$
The denominator of the fraction of the right-hand-side is positive.
Since $t^u\to1$ and $\lambda^u(\mu_{t^u})\to\log2$ as $b\to0$,
the first term of the numerator is $\geq(1/2)\sqrt{\varepsilon}$.
Hence the numerator is positive, and the first inequality holds.
A proof of the second one is analogous.
\end{proof}

 \noindent{\it Proof of Theorem A.} 
 Given a bounded interval $I\subset(-1,\infty)$,
 choose $\varepsilon$ and $b$ so that $I\subset (t_-,t_+)$.
Let $t\in I$.
In view of Corollary \ref{uniqueL}
we need to consider measures which are not liftable to the inducing scheme $(S,\tau)$.
Since
$$\sup\{F_{\varphi_t}(\mu)\colon\mu\in\mathcal M(f)\setminus\mathcal M_L(f)\}=
\sup\{F_{\varphi_t}(\mu)\colon\mu\in\mathcal M^e(f)\setminus\mathcal M_L(f)\},$$
we may restrict ourselves to ergodic measures.
It suffices to show
\begin{equation}\label{contra}
\sup\{F_{\varphi_t}(\mu)\colon\mu\in\mathcal M^e(f)\setminus\mathcal M_L(f)\}
<P_L(\varphi_t).\end{equation}

We argue by contradiction assuming (\ref{contra}) is false.
Then,
for any $\delta>0$ there
exists $\mu\in\mathcal M^e(f)\setminus\mathcal M_L(f)$ such that
$h(\mu)-t\lambda^u(\mu)\geq P_L(\varphi_t)-\delta.$
 Then
  $$h(\mu)\geq t\left(\lambda^u(\mu)-\lambda^u(\mu_{t^u})\right)
+t^u\lambda^u(\mu_{t^u})-\delta.$$
For the rest of the proof we deal with two cases separately.
\medskip

\noindent{\bf Case I:} $0\leq t<t_+$. We have
$h(\mu)\geq t\left(\lambda^u_m-\lambda^u(\mu_{t^u})\right)
+t^u\lambda^u(\mu_{t^u})-\delta.$
Since $\delta>0$ is arbitrary we get
\begin{equation}\label{t1''}
h(\mu)\geq t\left(\lambda^u_m-\lambda^u(\mu_{t^u})\right)
+t^u\lambda^u(\mu_{t^u}).\end{equation}
\eqref{t1''} and the first inequality in Lemma \ref{t1'} yield $h(\mu)>2\varepsilon$.
(P6) gives $\mu\in\mathcal M_L(f)$, which is
a contradiction.
\medskip

\noindent{\bf Case II:} $t_-<t<0$. Follows similarly from the second inequality in Lemma~\ref{t1'}.\ \qed


\subsection{Measure of maximal unstable dimension}\label{c}
We now prove the existence and uniqueness of a measure of maximal unstable dimension.\\

\noindent{\it Proof of Theorem C.} Let $\mu\in\mathcal M^e(f)$. If $\mu\in\mathcal M_L(f)$, 
then $\mu(\bigcup_{n\geq0} f^nX)=1$. Arguing similarly to the last paragraph in the proof of Theorem B we obtain $\dim^u(\mu)\leq d(X)=t^u$.
If $\mu\notin\mathcal M_L(f)$, then (P6) gives $h(\mu)<2\varepsilon$, and since $t^u\to1$ as $b\to 0$, we have
$\dim^u(\mu)<t^u$ for $b, \varepsilon$ small enough. 
Since $\dim^u(\mu_{t^u})=t^u$, 
$\mu_{t^u}$ is a measure of maximal unstable dimension.

As for the uniqueness, let $\mu$ be a measure of maximal unstable dimension.
Then $\dim^u(\mu)=t^u$,
and so $h(\mu)-t^u\lambda^u(\mu)=0$, namely $\mu$ is an equilibrium
measure for $\varphi_{t^u}$.
The uniqueness in Theorem A yields $\mu=\mu_{t^u}$.\qed
\medskip

\subsection{Statistical properties of equilibrium measures}\label{d}
We now prove statistical properties of $\mu_t$.\\

\noindent{\it Proof of Theorem D.}
Once the existence of the equilibrium measure for $\varphi_t$ is established, the statistical properties in Theorem D can be deduced from the abstract results of Young \cite[Theorems 2 and 3]{You98} with the exponential tail estimate in \eqref{exptail}. \qed

\section*{Appendix: computational proofs}
We refer the reader to \cite[Sect.2]{SenTak11} for relevant definitions and results used in this appendix.
 
 \subsection*{A1. Proof of Lemma \ref{markov}}
 Let $\zeta$ denote the critical point \cite[Sect.2.2]{SenTak11} on $\gamma$, and
$p(z)$ the corresponding {\it bound period} for $z\in\gamma_n$ \cite[Sect.2.3]
{SenTak11}. By \cite[Proposition 2.5]{SenTak11}, $\|D_{z}f^{p(z)}|E^u_z\|\geq
(4-2\varepsilon)^{\frac{p(z)}{2}}$ and ${\rm
slope}(D_zf^{p(z)}|E_z^u)\leq\sqrt{b}$. Since $p(z)<n$, 
the derivative estimate in \cite[Lemma 2.3]{SenTak11} gives
$\| D_{f^{p(z)}z}f^{n-p(z)}|E_{f^{p(z)}z}^u\|\geq\sigma_1^{n-p(z)}$.
Hence $\|D_zf^n|E^u_z\|\geq  (4-2\varepsilon)^{\frac{p(z)}{2}}\sigma_1^{n-p(z)}\geq
\sigma_1^n$ and (a) holds.
\medskip


For $z\in\gamma_n$, Let $e^u(z)$ denote the unit vector which spans $E^u_z$
and has a positive first component.
Consider the {\it stable foliation} $\mathcal F^s$ \cite[Sect.2.2]{SenTak11}, and
let $\mathcal F^s(fz)$ denote the leaf through $fz$.
Let $e^s(fz)$ denote the unit vector which spans
$T_{fz}\mathcal F^s(fz)$ and has a positive second component.
 Split
$ D_zfe^u(z)
= A(z)\left(
\begin{smallmatrix}1\\0
\end{smallmatrix}\right)+B(z)e^s(fz).$ 
\cite[Lemma 2.2]{Tak12b} gives
\begin{equation}\label{quadratic}
|A(z)|\approx|\zeta-z|\ \ \text{and} \ \ |B(z)|\leq C\sqrt{b}.
\end{equation}
Let $p=\max\{p(z)\colon z\in\gamma_n\}$. Split $\Vert D_xf^{p}e^u(x)-D_yf^{p}e^u(y)\Vert\leq I_1+I_2+I_3+I_4$,
where
\begin{align*}
I_1&=|A(x)-A(y)|\cdot \Vert 
D_{fx}f^{p-1}\left(
\begin{smallmatrix}1\\0
\end{smallmatrix}\right)\Vert,\\
I_2&=|B(x)-B(y)|\cdot \Vert D_{fx}f^{p-1}e^s(fx)\Vert
,\\
I_3&=|B(y)|\cdot \Vert D_{fx}f^{p-1}e^s(fx)-
D_{fy}f^{p-1}e^s(fy)\Vert,\\
 I_4&=|A(y)|\cdot \Vert D_{fx}f^{p-1}\left(
\begin{smallmatrix}1\\0
\end{smallmatrix}\right) -D_{fy}f^{p-1}\left(
\begin{smallmatrix}1\\0
\end{smallmatrix}\right)\Vert.
\end{align*}

\noindent{\it Estimates of $I_1, I_2$.} Let
$e^s(z)=\left(\begin{smallmatrix} e_1(z)\\e_2(z)\end{smallmatrix}\right)$, and 

$$S(z)=\begin{pmatrix}1 & e_1(z)\\0 & e_2(z)\end{pmatrix}^{-1}=\begin{pmatrix}1+\epsilon_1 &\epsilon_2\\
\epsilon_3 &1+\epsilon_4\end{pmatrix}\ \ \text{and} \ \
 D_zf=\begin{pmatrix}-2a^*z_1+\alpha_1 & \alpha_2\\
\alpha_3&\alpha_4\end{pmatrix},$$ 
where $z_1$ denotes the first coordinate of $z$.
Let $R(z)$ denote the rotation matrix by 
$\theta(z):=\angle(e^u(z),\left(\begin{smallmatrix} 1\\0\end{smallmatrix}\right))$. 
Then $A(z)$, $B(z)$ are equal to
the $(1,1)$, $(2,1)$ entries of the matrix $S(z)\cdot D_zf\cdot R(z)^{-1}$
correspondingly.
 A direct computation shows that
$A(z)$, $B(z)$ are linear combinations of $\alpha_i,\epsilon_i$ ($1\leq
i\leq 4$), $\cos\theta$, $\sin\theta$, all of which are Lipschitz
continuous on $\gamma_n$, from (\ref{henon}), property (F3) of $\mathcal F^s$ in
 \cite[Sect.2.2]{SenTak11}
and the
$C^2(b)$-property of $\gamma_n$. Hence $A(z)$, $B(z)$ are
Lipschitz continuous on $\gamma_n$ as well, which implies 
\begin{equation}\label{I12}  I_1\leq C|x-y|\cdot\|w_p(\zeta)\|
\quad\text{and}\quad I_2\leq (Cb)^{p-1}|x-y|
.\end{equation}

\noindent{\it Estimate of $I_3$.}
We start with an elementary geometric reasoning.
Let $v_1$, $v_2$ be nonzero vectors in $\mathbb R^2$ such that $\|v_1\|\leq\|v_2\|$, $\theta\ll1$ (See figure 3). We have
\begin{align*}
\|v_2-v_1\|&<|\|v_2\|-\|v_1\|\cos\theta|+\|v_1\|\sin\theta\\
&=\cos\theta|\|v_2\|-\|v_1\||+(1-\cos\theta)\|v_2\|+\|v_1\|\sin\theta\\
&\leq |\|v_2\|-\|v_1\||+2\theta\|v_2\|.
\end{align*}

Without loss of generality we may assume 
$\|D_yf^{p-1}e^s(y)\|\geq \|D_xf^{p-1}e^s(x)\|.$
The angle between the two vectors involved in $I_3$ is small.
The fact that $|B(y)|\leq C$
and the above reasoning show
 \begin{equation}\label{I3'}I_3\leq C\|D_yf^{p-1}e^s(y)\|\left(
\left|\frac{\|D_xf^{p-1}e^s(x)\|}{\|D_yf^{p-1}e^s(y)\|}-1\right|+3\|e^s(f^px)-e^s(f^py)\|\right).\end{equation} 

To estimate the first term in the parenthesis of \eqref{I3'} we argue as follows.
Let $J^s(z)=\|D_zfe^s(z)\|$. The invariance of the
stable foliation $\mathcal F^s$ gives
\begin{equation}\label{i3eq}\log\frac{\|D_xf^{p-1}e^s(x)\|}
{\|D_yf^{p-1}e^s(y)\|}\leq \sum_{i=1}^{p-1}\log\frac{J^s(f^ix)}
{J^s(f^iy)}.\end{equation} 
Let ${e^s}^\bot(z)$ denote any unit vector orthogonal to $e^s(z)$,
 $\theta(z)=\angle(D_zfe^s(z),D_zf{e^s}^\bot(z))$,
and let ${J^s}^\bot(z)=\|D_zf{e^s}^\bot(z)\|$. 
Then  ${e^s}^\bot$ and $\theta$ are Lipschitz continuous, 
$\theta\approx\pi/2$ and ${J^s}^\bot>2$.
Hence $\log {J^s}^\bot$ and $\sin\theta$ are Lipschitz continuous,
with Lipschitz constants independent of $b$.
Since\footnote{Here we use the fact that the Jacobian of the H\'enon map is constant equal to $b$.
Essentially the same argument remains to hold for H\'enon-like maps
for which there exists $C>0$ independent of $b$ such that $\|D\log|\det Df|\|\leq C$
(c.f. \cite{MorVia93}). Therefore our main theorems hold for H\'enon-like maps
satisfying this assumption.}
$J^s(f^ix) {J^s}^\bot(f^ix)\sin\theta(f^ix)=|\det D_{f^ix}f|=b,$
for $1\leq i<p$ we have
\[\log\frac{J^s(f^ix)} {J^s(f^iy)}
=\log\frac{{J^s}^\bot(f^iy)} {{J^s}^\bot(f^ix)}
+\log\frac{\sin\theta(f^iy)}{\sin\theta(f^ix)}\leq C|f^ix-f^iy|.\]

\begin{sublemma}\label{distance}
$\sum_{i=1}^{p-1}|f^ix-f^iy| \leq
C|f^px-f^py|+\frac{C}{\sqrt{b}}\frac{|x-y|}{d(\zeta,\gamma_n)}.$
\end{sublemma}
\begin{proof}
We introduce a new coordinate $(\xi,\eta)$ as follows:
for a point $z=(z_1,z_2)\in\mathbb R^2$ let
$$(z_1,z_2)=(\xi(z),0)+(\eta(z)e_1^s(f\zeta),\eta(z)e_2^s(f\zeta)).$$
Note that there exist $C_1>C_2>0$ such that
$$C_2\leq\frac{|z-z'|}{|\xi(z)-\xi(z')|+|\eta(z)-\eta(z')|}\leq C_1.$$

Parametrize $\gamma_n$ by arc length $s$ and let 
$\hat\gamma(s)=f(\gamma_n(s))$.
Let $x=\gamma_n(s_1)$, $y=\gamma_n(s_2)$, and assume $s_1<s_2$
without loss of generality.
For $s\in[s_1,s_2]$ define two numbers $\tilde A(s)$, $\tilde B(s)$ by
$\frac{d\hat\gamma}{dt}|_{t=s}=\tilde A(s)\left(\begin{smallmatrix}1\\0\end{smallmatrix}\right)+
\tilde B(s)e^s(f\zeta).$
\eqref{quadratic} and $|e^s(f\zeta)-e^s(fz)|\leq C|\zeta-z|$ together imply
$|\tilde A(s)|\approx|\zeta-\gamma_n(s)|$ and
$|\tilde B(s)|\leq C\sqrt{b}$.
Since $|\zeta-\gamma_n(s)|\geq d(\zeta,\gamma_n)$ we have
$|\xi(\gamma(s_1))-\xi(\gamma(s_2))|=|\int_{s_1}^{s_2}\tilde A(s)ds|\geq 
|s_1-s_2|Cd(\zeta,\gamma_n)$,
and
$|\eta(\gamma(s_1))-\eta(\gamma(s_2))|=|\int_{s_1}^{s_2}\tilde B(\gamma(s))ds|
\leq C\sqrt{b}|s_1-s_2|.$
Namely,
\begin{equation}\label{quadra}|\xi(fx)-\xi(fy)|\geq Cd(\zeta,\gamma_n)|x-y|\ \ \text{and} \ \
|\eta(fx)-\eta(fy)|\leq C\sqrt{b}|x-y|.\end{equation}
For $1\leq i\leq p-1$ we have
$|\xi(f^ix)-\xi(f^iy)|\leq \sigma_1^{-(i-p)}|\xi(f^px)-\xi(f^py)|$, and
\begin{align*}
|\eta(f^ix)-\eta(f^iy)|&\leq (Cb)^{\frac{i-1}{2}}|\eta(fx)-\eta(fy)|\leq
\frac{(Cb)^{\frac{i-1}{2}}}{\sqrt{b}}\frac{|\xi(fx)-\xi(fy)|}{d(\zeta,\gamma_n)}
\leq \frac{(Cb)^{\frac{i-1}{2}}}{\sqrt{b}}\frac{|x-y|}{d(\zeta,\gamma_n)},\end{align*}
where the second inequality follows from \eqref{quadra}. 
Summing these two inequalities over all $1\leq i\leq p-1$ yields the desired one.
\end{proof}
Sublemma \ref{distance} implies that the right-hand-side of \eqref{i3eq} is bounded by a constant $C>0$ independent of $b$.
Since there exists $\rho=\rho(C)>0$ such that
$e^\psi\leq1+\rho\psi$ for $0\leq\psi\leq C$, we have
\begin{equation}\label{i4eq}
\frac{\|D_xf^{p-1}e^s(x)\|} {\|D_yf^{p-1}e^s(y)\|}-1
 \leq      \rho\sum_{i=1}^{p-1}\log\frac{J^s(f^ix)} {J^s(f^iy)}\leq    \rho C\sum_{i=1}^{p-1}|f^ix-f^iy|.\end{equation} 
 
 We are in position to finish the estimate of $I_3$.
We have 
$\|D_yf^{p-1}e^s(y)\|\leq Cb$, and the second
term in the parenthesis in \eqref{I3'} 
is $\leq C|f^px-f^py|$. Then,
combining  Sublemma \ref{distance}, \eqref{i4eq} and 
plugging the result  into \eqref{I3'} yields
\begin{equation}\label{I3}I_3\leq
Cb|f^px-f^py|+C\frac{|x-y|}{d(\zeta,\gamma_n)}.\end{equation}

\begin{figure}
\begin{center}
\includegraphics[height=3cm,width=10cm]
{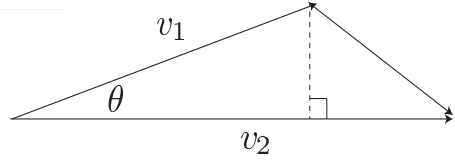}
\caption{$\|v_1\|\leq\|v_2\|$, $\theta\ll1$}
\end{center}
\end{figure}

\noindent{\it Estimate of $I_4$.} 
In the same way as in the proof of \eqref{I3'} we have
\[
I_4\leq
|A(y)|\cdot\Vert w_{p}(\zeta)\Vert\left(
\left|\frac{\|D_{fx}f^{p-1}\left(\begin{smallmatrix}1\\0\end{smallmatrix}\right)
\|}{\|D_{fy}f^{p-1}\left(\begin{smallmatrix}1\\0\end{smallmatrix}\right)
\|}-1\right|+2 \angle(D_{fx}f^{p-1}\left(\begin{smallmatrix}1\\0\end{smallmatrix}\right),
D_{fy}f^{p-1}\left(\begin{smallmatrix}1\\0\end{smallmatrix}\right))\right).
\]
From the distortion estimate in the proof of 
 \cite[Lemma 2.7]{SenTak11} and Sublemma \ref{distance},
 the first term in the parenthesis is $\leq C|f^px-f^py|$.
To estimate the second term in the parenthesis, take a point $r$ so that the leaf
$\mathcal F^s(fy)$ intersects the horizontal through $fx$
at $fr$. 
By the angle estimate in \cite[Claim 5.3]{WanYou01},
$$\angle (D_{fy}f^{p-1}\left(\begin{smallmatrix}1\\0\end{smallmatrix}\right),
D_{fr}f^{p-1}\left(\begin{smallmatrix}1\\0\end{smallmatrix}\right))\leq (Cb)^{p-1}|fy-fr|\leq(Cb)^{p-1}
|x-y|\leq (Cb)^{p-1}|f^px-f^py|.$$
By the $C^2(b)$-property and the definition of $r$,
$$\angle (D_{fx}f^{p-1}\left(\begin{smallmatrix}1\\0\end{smallmatrix}\right),
D_{fr}f^{p-1}\left(\begin{smallmatrix}1\\0\end{smallmatrix}\right))\leq\sqrt{b}|f^{p}x-f^{p}r|\leq
C\sqrt{b}|f^{p}x-f^py|.$$
Hence we obtain
\begin{equation*}\angle(D_{fx}f^{p-1}\left(\begin{smallmatrix}1\\0\end{smallmatrix}\right),
D_{fy}f^{p-1}\left(\begin{smallmatrix}1\\0\end{smallmatrix}\right))\leq C\sqrt{b}|f^{p}x-f^py|.\end{equation*}
Additionally \eqref{quadratic} yields
$$|A(y)|\leq C|\zeta-y|\leq
C(d(\zeta,\gamma_n)+\ell(\gamma_n))
$$
where $d(\cdot,\cdot)$ denotes the minimal distance apart. Finally,
from Sublemma \ref{length} below we get 
\begin{equation}\label{I4}
I_4\leq Cd(\zeta,\gamma_n)\|w_p(\zeta)\|\cdot |f^px-f^py|.
\end{equation}
\begin{sublemma}\label{length}
$\ell(\gamma_n)\leq Cd(\zeta,\gamma_n).$
\end{sublemma}
\begin{proof}
Let $M$ be a large integer such that $M\ll N$. Consider the leaf
of the stable foliation $\mathcal F^s$
through $f\zeta$ which is of the form $\mathcal F^s(f\zeta)=\{(x(y),y)\colon
|y|\leq\sqrt{b}\}$. For $k>M$ define 
$$U_k:=\left\{(x,y)\colon
D_{k}\leq |x-x(y)|< D_{k-M},|y|\leq\sqrt{b}\right\}$$
where $D_k:=C\left[\sum _{i=1}^{k}\frac{\|w_{i}(\zeta)\|^2}{\|w_{i+1}(\zeta)\|}\right]^{-1}$ for some constant $C>0$. 
Let 
$k_0:=\max\{k>M\colon U_k\cap f\gamma_n\neq\emptyset\}-1$. 
By  \cite[Lemma 2.5(a)]{SenTak11}, there exist constants $0<C_1<C_2<1/2$ such that
\begin{equation}\label{excercise}
C_1D_{k_0-M}\leq D_{k_0}\leq C_2D_{k_0-M}.\end{equation}
We prove
\begin{equation}\label{contain}f\gamma_n\subset U_{k_0}\cup U_{k_0+1}.\end{equation}
(\ref{excercise}) \eqref{contain} 
imply
$\ell(\gamma_n)\leq C\sqrt{D_{k_0-M}}\leq C\sqrt{D_{k_0}}\leq Cd(\gamma_n,\zeta)$,
and thus Sublemma \ref{length} holds. 
\medskip

It is left to prove \eqref{contain}.
If the inclusion were false, then one could choose a curve
$\delta\subset f\gamma_n\cap U_{k_0}$ with endpoints in
the two vertical boundaries of $U_{k_0}$. 
Let $x$ denote the endpoint of $f^2\gamma_n$ in
$\tilde\alpha_{n-1}$. The bounded distortion and  the second inequality in 
\cite[Lemma 2.5(b)]{SenTak11} give
$$d(\alpha_0^-,f^{k_0-M}x)\leq 2D_{k_0} \|w_{k_0-M+1}(\zeta)\|\leq
2\cdot3^{-M}D_{k_0}\|w_{k_0}(\zeta)\|\leq 3^{-M},$$ 
and
$$\ell(f^{k_0-M}\delta)\geq C(D_{k_0-M}-D_{k_0})\|w_{k_0-M}(\zeta)\|\geq C(1-
C_2)D_{k_0-M}
\|w_{k_0-M}(\zeta)\|\geq C.$$
From these
two estimates and choosing large $M$ if necessary we have that the interior of $f^{k-10}\delta$ intersects 
some $\tilde\alpha_i$. This yields a contradiction.  \end{proof}

\noindent{\it Overall estimates.} Gluing \eqref{I12} \eqref{I3} \eqref{I4}
together,
$$\Vert D_xf^pe^u(x)-D_yf^pe^u(y)\Vert\leq
C\|w_{p}(\zeta)\| \cdot|x-y|+\frac{C|x-y|}{d(\zeta,\gamma_n)} +Cd(\zeta,\gamma_n)\|w_{p}(\zeta)\|\cdot
|f^px-f^py|.$$ 
From the proof of \cite[Proposition 2.6]{SenTak11} there exists $C>0$
such that $\|D_zf^pe^u(z)\|\geq C
d(\zeta,\gamma_n)\cdot\|w_p(\zeta)\|\geq1$
for  $z=x,y$. Hence
\begin{equation}\label{recov4}
\log\frac{\|D_xf^{p}|E^u_x\|}{\|D_yf^{p}|
E^u_y\|}\leq\frac{C|x-y|}{d(\zeta,\gamma_n)}+C|f^px-f^py|.
\end{equation}
 \cite[Proposition 2.5(c)(d)]{SenTak11} and \cite[Lemma 2.4]{WanYou01} together imply that
 $f^p\gamma_n$ is a $C^2(b)$-curve. By the uniform hyperbolicity outside of $\Theta$,
$f^{p+1}\gamma_n,\ldots,f^{n-1}\gamma_n$ are
$C^2(b)$ as well and we have 
\begin{equation}\label{recov5}|f^px-f^py|\leq  |f^nx-f^ny|\ \ \text{and} \ \
\log\frac{\|D_{f^px}f^{n-p}|E^u_{f^px}\|}{\|D_{f^py}f^{n-p}|E^u_{f^py}\|}\leq C|f^nx-f^ny|.
\end{equation} 
\eqref{recov4} \eqref{recov5} and Sublemma \ref{recov} below yield
$$\log\frac{\|D_xf^{n}|E^u_x\|}{\|D_yf^{n}|
E^u_y\|}=\log\frac{\|D_xf^{p}|E^u_x\|}{\|D_yf^{p}|
E^u_y\|}+\log\frac{\|D_{f^px}f^{n-p}|E^u_{f^px}\|}{\|D_{f^py}f^{n-p}|E^u_{f^py}\|}\leq C|f^nx-f^ny|,$$
which proves (b). 

\begin{sublemma}\label{recov}
$\frac{|x-y|}{d(\zeta,\gamma_n)}\leq C|f^nx-f^ny|$.
\end{sublemma}
\begin{proof}

By the bounded
distortion outside of $\Theta$, there exists
$\theta\in f\gamma_n$ such that
\begin{equation}\label{reco1}|\xi(fx)-\xi(fy)|\cdot\|D_{\theta}f^{n-1}\left(\begin{smallmatrix}1\\0\end{smallmatrix}\right)\| \leq C |f^nx-f^ny|.\end{equation}
The bounded distortion outside of $\Theta$ and 
 the quadratic behavior near $\zeta$ as in
\eqref{quadratic} imply
$$\ell(\gamma_n)
d(\zeta,\gamma_n)\|D_{\theta}f^{n-1}\left(\begin{smallmatrix}1\\0\end{smallmatrix}\right)\|
\geq C\ell(f^n\gamma_n).$$
Hence there exists $C>0$ such that
\begin{equation}\label{reco2}
d(\zeta,\gamma_n)^2\|D_{\theta}f^{n-1}\left(\begin{smallmatrix}1\\0\end{smallmatrix}\right)\|\geq C\ell(\gamma_n)
d(\zeta,\gamma_n)\|D_{\theta}f^{n-1}\left(\begin{smallmatrix}1\\0\end{smallmatrix}\right)\|\geq C\ell(f^n\gamma_n)>C.\end{equation}
The first inequality follows from Sublemma \ref{length}, and
the last inequality is because $f^n\gamma_n$ is a $C^2(b)$-curve with endpoints
in $\alpha_1^\pm$.
Using the first inequality and \eqref{quadra} and then
\eqref{reco1} \eqref{reco2} yield
\begin{equation*}\label{recoveq6}
\frac{|x-y|}{d(\zeta,\gamma_n)}\leq\frac{C|\xi(fx)-\xi(fy)|}{d(\zeta,\gamma_n)^2}
\leq \frac{C|f^nx-f^ny|}{d(\zeta,\gamma_n)^2
\|D_{\theta}f^{n-1}\left(\begin{smallmatrix}1\\0\end{smallmatrix}\right)\|}
\leq C |f^nx-f^ny|.\qedhere
\end{equation*}
\end{proof}

\subsection*{A2. Proof of Lemma \ref{leaf}.}
Set $\kappa=5^{-(1+\xi)N}$.
\begin{sublemma}\label{appendix2}
For any $z\in\Omega_\infty$ and every $n\geq1$, $\|D_zf^n|E^u_z\|\geq\kappa^{n}$. 
\end{sublemma}
\begin{proof}
With the terminology in \cite[Sect.2.5]{SenTak11} we introduce the bound/free 
structure on the orbit of $z$, 
using $\Theta_0$ as a critical neighborhood. If $f^{n}z$ is free, then 
the orbit $z,\ldots,f^nz$ is decomposed into alternative bound and free segments. Applying the expansion estimates in 
\cite[Lemma 2.3, Proposition 2.5]{SenTak11}
alternatively we have $\Vert D_zf^n|E^u_z\Vert\geq \kappa^{n}$. If $f^{n}z$ is bound, then
 there exists an integer $0<m<n$ such that
$f^mz\in\Theta_0$ and $m<n<m+p$, where $p$ is the bound period of $f^mz$.
Since $f^{m+p}z$ is free and $\|Df\|<5$ we have
$\Vert D_zf^n|E^u_z\Vert\geq 5^{-(m+p-n)} \Vert
D_zf^{m+p}|E^u_z\Vert> 5^{-p},$ and since $z\in\Omega_\infty$ we have $p\leq \xi m+N\leq \xi n+N$ 
and so $\|D_zf^n|E^u_z\|\geq 5^{- \xi n-N }\geq \kappa^{n}.$
\end{proof}
From Sublemma \ref{appendix2} and the results in \cite[Sect.6, Sect.7C]{MorVia93}, 
there exists a long stable leaf through $z$. The uniqueness follows from the next
sublemma with $n=0$.
\begin{sublemma}\label{appendix2'}
Let $z_1,z_2\in\Omega_\infty$ and let $\gamma^s(z_i)$ denote any long stable
leaf through $z_i$ $(i=1,2)$. 
If $f^n(\gamma^s(z_1))\cap \gamma^s(z_2)\neq\emptyset$ for some $n\geq 0$,
then $f^n\gamma^s(z_1)\subset\gamma^s(z_2)$.
\end{sublemma}
\begin{proof}
Choose a large integer $M\gg n$ such that $f^Mz_2$ is free. Take $x_1\in
f^n\gamma^s(z_1)$, $x_2\in\gamma^s(z_2)$ which are connected by
a horizontal segment of length $b^{\frac{M}{3}}$. By construction,
$\|D_{x_2}f^M|E^u_{x_2}\|\geq
\sigma_1^{M}$. By the bounded distortion, 
the $f^M$-iterate of the segment is $C^2(b)$ and
$|f^Mx_1-f^Mx_2|\geq C\sigma_1^{M}|x_1-x_2|
\geq  C\sigma_1^{M}b^{\frac{M}{3}}.$
If $q\in f^n\gamma^s(z_1)\cap\gamma^s(z_2)$,
then 
$|f^Mx_1-f^Mx_2|\leq|f^Mx_1-f^Mq|+
|f^Mq-f^Mx_2|\leq 2(Cb)^{\frac{M}{2}}.$ These two estimates
are incompatible. 
\end{proof}
Lemma \ref{leaf}(a) is a consequence of Sublemma \ref{appendix2'}.
Since the stable sides of $\Theta$ are long stable leaves,
$\gamma^s(z)\subset\Theta$ follows from the disjointness in Lemma \ref{leaf}(a).
The rest of the items in the lemma follow from the
results in \cite[Sect.6, Sect.7C]{MorVia93} and \cite[Proposition 2.4]{BenVia01}.
$\hfill\qed$

\subsection*{A3. Proof of Lemma \ref{Lyap}.}
Let $\mu\in\mathcal M^e(f)$. 
Consider $x\in \Omega$ which is free and satisfies $\displaystyle{\lim_{n\to\infty}}(1/n)\log \|D_xf^n|E^u_x\|=\lambda^u(\mu)$.
The orbit $x,fx,\ldots$ is decomposed into alternative bound and free segments. Applying the expansion estimates in 
\cite[Lemma 2.3, Proposition 2.5]{SenTak11}
alternatively we have $\Vert D_xf^n|E^u_x\Vert\geq (2-\varepsilon)^{n}$
if $f^nx$ is free. This implies
$\lambda^u(\mu)\geq\log(2-\varepsilon).$
$\hfill\qed$

\subsection*{Acknowledgments}
We thank Renaud Leplaideur, Isabel Rios, Juan Rivera-Letelier, Christian Wolf and anonymous referees for very useful comments.
S.S. is partially supported by the CNPq and PRONEX, Brazil.
H.T. is partially supported by the Grant-in-Aid for Young Scientists (B) of the JSPS, Grant No.2374012 and the Keio Gijuku Academic Development Funds.
This research is partially supported by the Kyoto University Global COE Program.
We thank the Mathematics Departments of Kyoto University, the Federal University of Rio de Janeiro, the Pennsylvania State University, 
l'\'Ecole Polytechnique F\'ed\'erale de Lausanne, and IMPA for their hospitality.

\end{document}